\documentclass[11pt,reqno]{amsart}

\usepackage[latin1]{inputenc}
\usepackage{amsmath,amsthm,amssymb,amscd}
\usepackage{mathrsfs,mathtools,color}

\usepackage{hyperref}

\definecolor{darkgreen}{rgb}{0.0, 0.6, 0.13}
\definecolor{airforceblue}{rgb}{0.0, 0.30, 0.69}

\newtheorem{thm}{Theorem}[section]
 
 \newtheorem{lem}[thm]{Lemma}
 \newtheorem{prop}[thm]{Proposition}

 \theoremstyle{definition}
 \newtheorem{df}[thm]{Definition}
 \theoremstyle{remark}
 \newtheorem{rem}[thm]{Remark}
 
 \numberwithin{equation}{section}

\def\sqw{\hbox{\rlap{\leavevmode\raise.3ex\hbox{$\sqcap$}}$%
\sqcup$}}
\def\findem{\ifmmode\sqw\else{\ifhmode\unskip\fi\nobreak\hfil
\penalty50\hskip1em\null\nobreak\hfil\sqw
\parfillskip=0pt\finalhyphendemerits=0\endgraf}\fi}

\newcommand{\R}{\mathbb R}

\newcommand{\T}{{\bf T}}

\begin{document}

\title[Optimal local well-posedness for the DNLS equation]{Optimal local well-posedness for the periodic derivative nonlinear Schr\"{o}dinger equation}
\author[Yu Deng]{Yu Deng }
\address{Department of Mathematics, University of Southern California, Los Angeles,  CA 90089, USA}
\email{yudeng@usc.edu}
\author[Andrea R. Nahmod]{Andrea R. Nahmod}
\address{Department of Mathematics, University of Massachusetts, 710 N.\ Pleasant Street, Am\-herst, MA 01003 USA}
\email{nahmod@math.umass.edu}
\author[Haitian Yue]{Haitian Yue}
\address{Department of Mathematics, University of Southern California, Los Angeles,  CA 90089, USA}
\email{haitiany@usc.edu}

\subjclass[2010]{35, 42}
\thanks{Andrea R. Nahmod is partially supported by NSF-DMS-1463714 and NSF-DMS-1800852.}

\begin{abstract}

We prove local well-posedness for the periodic derivative nonlinear Schr\"odinger's equation, which is $L^2$ critical, in Fourier-Lebesgue spaces which scale like ${H}^s(\mathbb T)$ for $s>0$. In particular we close the existing gap in the subcritical theory by improving the result of Gr\"unrock and Herr \cite{GH}, which established local well-posedness in Fourier-Lebesgue spaces which scale like ${H}^s(\mathbb T)$ for $s >\frac{1}{4}$. We achieve this result by a delicate analysis of the structure of the solution and the construction of an adapted nonlinear submanifold of a suitable function space. Together these allow us to construct the unique solution to the given subcritical data.  This constructive procedure is inspired by the theory of para-controlled distributions developed by Gubinelli-Imkeller-Perkowski \cite{GIP} and  Cantellier-Chouk \cite{CCh} in the context of stochastic PDE. Our proof and results however, are purely deterministic. 
\end{abstract}

\maketitle

\section{Introduction} 

The derivative nonlinear Schr\"odinger's equation
\begin{equation} \label{dnls} i u_t  \, + \,  \partial^2_{x} u  \, = i \partial_x (|u|^2 u), 
\end{equation}
where $(t,x) \in (-T,T)\times\mathbb T$ (periodic) or $(-T, T)\times \mathbb R$ (non-periodic),  is a Hamiltonian PDE introduced as a model for the propagation of nonlinear waves in plasma physics and nonlinear optics 
\cite{SS}. It is well-known as a completely integrable system \cite{KaupNew, HayOz2, Tak1, Herr}, and in particular conserves mass and energy.  The Cauchy problem for (\ref{dnls}) is scale invariant for data in $L^2$, that is, if $u(t,x)$ is a solution then so is $u_{\lambda}(t, x) = \lambda^{\frac{1}{2}} u(\lambda^2 t, \lambda x)$ with the same $L^2$ norm. Thus {\it a priori} one expects local well-posedness for \eqref{dnls} with initial data data in $H^{s}$ for $s \geq 0$. However, while local well-posedness in $H^s$ for (\ref{dnls}) is known for $s \ge \frac{1}{2}$ \cite{Tak1,Herr}, one has \emph{ill-posedness} in $H^s$ for $s< \frac{1}{2}$  \cite{BiaLin, Tak1, Herr}.

One way to close the gap between the scaling heuristics and actual local well-posedness results is by considering data in the Fourier-Lebesgue spaces $H^\sigma_p$, where $p\geq 2$. These spaces are defined as
\begin{equation}\label{fourierlebesgue} \Vert u_0 \Vert_{ H_p^\sigma } :=   \|\langle k  \rangle^{\sigma}\widehat u_0(k) \|_{L_k^p}  \end{equation} (with $L_k^p$ replaced by $\ell_k^p$ in the periodic case). These spaces have naturally arisen in the literature and we refer the reader to  e.g. \cite{Hormander, VarVe, Gr, Christ, Gr-wave, GrNa} for some instances. 
Note in particular, that in one dimension the Fourier-Lebesgue space $H_p^\sigma$ has the same scaling{\footnote{Here and henceforth we mean that the homogenous part of the Fourier-Lebesgue norm scales like the corresponding homogeneous Sobolev norm}} as the Sobolev space $H^s$ for \begin{equation}\label{scale}s=\sigma+\frac{1}{p}-\frac{1}{2}; \end{equation} in particular, $H^{\frac{1}{2}}_{\infty}$ has the scaling of $L^2$, and $H^{\frac{1}{2}}_2 = H^{\frac{1}{2}}$.

In the non-periodic case, Gr\"unrock \cite{Gr} proved optimal local well-posedness for (\ref{dnls}) in $H^\sigma_p(\mathbb{R})$ for $\sigma\geq \frac{1}{2}$ and $p<\infty$, which allows the corresponding Sobolev regularity $s$ to be arbitrarily close to $0$, thus covering the full subcritical range. The proof combines the gauge transformation introduced in \cite{Hay} (used also in \cite{HayOz1, HayOz2, Tak1}) and new bilinear and trilinear estimates for the gauged equation in an appropriate variant of Bourgain's Fourier restriction norm spaces \cite{B8} (see Section \ref{function} below for details) which follow from the dispersion and the smoothing properties of the Schr\"odinger propagator on $\mathbb R$.

In the periodic case, however, local well-posedness for (\ref{dnls}) in $H_p^\sigma(\mathbb{T})$ is only known for $\sigma\geq \frac{1}{2}$ and $2 \leq p<4$, which scaling-wise correspond to Sobolev regularity $s > \frac{1}{4}$. This is the work of Gr\"unrock and Herr \cite{GH}. Their proof is based on the adapted periodic gauge transformation in \cite{Herr} and new multilinear estimates for the gauged equation in adapted variants of the Fourier restriction norm spaces.  Moreover, it is proved in \cite{GH} that the crucial multilinear estimates{\footnote{More precisely the trilinear estimates containing as one of their inputs the derivative term.}} become false when $p\geq 4$, so this result as well as the existing gap in the local well-posedness theory between $s>1/4$ and the scaling prediction $s>0$, cannot be improved \emph{within the framework of} \cite{GH}.
\smallskip

In this paper we close this existing gap in the periodic case. More precisely, we prove optimal local well-posedness for (\ref{dnls}) in $H_p^{\sigma}(\mathbb{T})$ for $\sigma\geq\frac{1}{2}$ and $p<\infty$ which covers the entire subcritical regime, hence yielding optimal local well-posedness. Our main theorem is stated as follows:
\begin{thm}\label{main} Fix $\sigma\geq\frac{1}{2}$ and $p_0<\infty$. For any $A>0$, there exists $T=T(p_0,A)>0$, such that if $\|u_0\|_{H_{p_0}^{\sigma}}\leq A$, then there exists a unique solution $u\in \mathcal{Z}\subset C_t^0H_{p_0}^{\sigma}(J)$ to (\ref{dnls}) with initial data $u(0)=u_0$, where $J=[-T,T]$. Here $\mathcal{Z}$ is an explicitly defined sub-manifold of $C_t^0H_{p_0}^{\sigma}(J)$, see Definition \ref{defset} below. The map $u_0\mapsto u$ is continuous with respect to the $C_t^0H_{p_0}^{\sigma}(J)$ metric.
\end{thm} 
\begin{rem}\label{property} The solution we construct solves \eqref{dnls} in the sense that it solves the integral equation
\begin{equation}\label{intg}u(t)=e^{it\partial_x^2}u_0+\int_0^te^{i(t-s)\Delta}\partial_x(|u(s)|^2u(s))\,\mathrm{d}s.\end{equation} 
 It is also the unique limit of smooth solutions: given $A>0$, and any smooth initial data $u_0$ in the $A$-ball of $H_{p_0}^{\sigma}$, the classical solution exists for time $T=T(p_0,A)>0$, and the data-to-solution map extends continuously to all of this ball. Moreover, if $p_0<4$, our solution coincides with the solution constructed in \cite{GH}, for as long as the latter exists.
\end{rem}
\begin{rem} We will only prove Theorem \ref{main} with $\sigma=\frac{1}{2}$ and $p_0\geq 4$. The extension to $\sigma>\frac{1}{2}$ is standard (see Proposition \ref{pres} for a sketch), and when $2 \leq p_0<4$ the result follows directly from \cite{Herr}.
\end{rem}

\begin{rem} Global-posedness for the Cauchy problem for \eqref{dnls} is known to hold for data in $H^s, \, s \geq \frac{1}{2}$ both on $\mathbb R$ \cite{ CKSTT1, CKSTT2, MWX} and on $\mathbb T$ \cite{ Herr, Win, Mos}. Furthermore, one has almost sure global well posedness for data in Fourier Lebesgue spaces $H^{\sigma}_p (\mathbb T)$ that have the scaling of $H^{\frac{1}{2}-\varepsilon}(\mathbb T), \, \varepsilon >0$ \cite{NORS, NRSS}.  In this paper our primary goal is to close the gap in the local well-posedness Cauchy theory.  One may then study the question of deterministic global well-posedness below $H^{\frac{1}{2}}(\mathbb T)$ 
which requires quite different techniques such as for example exploiting the integrability of the equation and seeking suitable new conservation laws below $H^{\frac{1}{2}}$.
\end{rem}

\subsection{The standard approach, and difficulties}
Generally speaking the difficulty one faces in solving \eqref{dnls} is a derivative loss arising from the term $ {i |u|^2 \, u_x } $ in the nonlinearity of \eqref{dnls}, and hence for 
low regularity data the key is to somehow make up for this loss. The first step towards this goal is a gauge transformation \cite{Hay, HayOz1, HayOz2, Tak1, Herr} which removes this bad resonant term in the nonlinearity that loses derivatives and makes the estimates uncontrollable. Matters are then reduce to studying the gauged derivative nonlinear Schr\"odinger equation which we schematically write as
\begin{equation}\label{gdnls}
(\partial_t-i\partial_x^2)v=\mathcal{C}(v,v,v)+\mathcal{Q}(v,\cdots,v),
\end{equation}
where the nonlinearity $\partial_x (|u|^2 u)$ has been transformed into the sum of the `better' cubic term \[\mathcal{C}(v,v,v) \sim  {i \overline{v}_x\cdot v^2  },\] plus a quintic term which contains no derivatives terms in it and which we momentarily neglect in this discussion as being `lower order'.   Once the bad nonlinear term is gauged away from \eqref{dnls}, the solution $v$ to (\ref{gdnls}) is constructed by a fixed point argument, which follows from proving multilinear estimates in suitably Fourier restriction norm function spaces adapted to the data space. In the non-periodic case \cite{Gr}  these spaces in conjunction with 
the dispersion and smoothing effects available on $\mathbb{R}$
suffice, as we mentioned above, to prove optimal local well-posedness{\footnote{ Local  well-posedness for the gauged equation \eqref{gdnls} implies local existence, uniqueness and continuity of the flow map for \eqref{dnls} \cite{Herr, GH}.}} for \eqref{gdnls} and hence for \eqref{dnls} in $H_p^\sigma$, where $\sigma\geq\frac{1}{2}$ and $p<\infty$.  In the periodic case \cite{GH}, however,  the authors  need to introduce a fourth parameter $q$ in the Fourier restriction norm function spaces, namely they define{\footnote{ When $p=q=2$ these spaces coincide with Bourgain's Fourier restriction norm spaces associated to the Schr\"odinger equation, and are simply denoted by $X^{s, b}$.}}
\[\|u\|_{X_{p,q}^{\sigma,b}}=\|\langle k\rangle^{\sigma}\langle\xi+k^2\rangle^b\widehat{u}(k,\xi) \|_{\ell_{k}^pL_\xi^q}\] and prove that, if $\sigma=\frac{1}{2}$ and $p<4$  the trilinear estimate
\begin{equation}\label{trilinear-ce} 
\bigg\| \int_0^t   e^{i (t-s) \partial_x^2 }  (\partial_x \overline{v}_1  \cdot v_2v_3) \, \mathrm{d}s \bigg\|_{ X^{\frac{1}{2}, b}_{p, q}} \lesssim  \prod_{j=1}^3  \| v_j\|_{ X^{\frac{1}{2}, b}_{p, q}} 
\end{equation} holds true  for $(b,q)=(\frac{1}{2}+,2)$. Furthermore, they construct explicit counterexamples showing  that  for $\sigma=\frac{1}{2}$ and $p\geq 4$, the trilinear estimate (\ref{trilinear-ce}) fails for \emph{any} choice of $(b,q)$  \cite{GH}.  In other words, when $\sigma=\frac{1}{2}$ and $p\geq 4$, which in the Sobolev scale corresponds to regularity $ 0< s\leq \frac{1}{4}$,  the local solution to (\ref{gdnls}) cannot be constructed directly by contraction mapping{\footnote{In principle, it might be possible that a trilinear estimate holds in some exotic Banach space not of form $X_{p,q}^{s,b}$ but, if not unlikely, this would at least require a rather sophisticated construction.}}. 

To prove our Theorem \ref{main} we must and will take a different approach.  After performing the gauge transformation,  our point of departure is the following observation:  let $\sigma=\frac{1}{2}$ and $p\geq 4$. If one compares the profiles of the counterexamples constructed in \cite{GH} with the profiles of the terms occurring in the formal Picard iterations of (\ref{gdnls}), then they will never coincide, although they belong to exactly the same  adapted Fourier restriction norm $X_{p,q}^{\sigma,b}$ spaces. Therefore it is reasonable to imagine that, the solution $v$ to (\ref{gdnls}) still exists in one of these spaces --say $b=\frac{1}{2}+$ and $q=2$ for definiteness-- but will have some \emph{specific structure} such that it precisely avoids the counterexamples constructed in \cite{GH}. 
To that effect we will construct $v$ in a \emph{nonlinear submanifold} $\mathcal{W}$ of the Banach space $X_{p,2}^{\sigma,\frac{1}{2}+}$ containing functions of a specific structure whence the trilinear estimate (\ref{trilinear-ce}) will actually hold true $\sigma=\frac{1}{2}$ and $ p \geq 4 $.  The heart of this paper will be to identify this precise structure. 

In order to motivate our approach we take a step back and review some of the methods developed in the probabilistic (random data, or stochastically forced) context. We  note in passing that an immediate corollary of our main Theorem \ref{main} is that for random initial data of form 
\begin{equation}\label{random-data}  u^{\omega}(0) := \sum_{k \in \mathbb Z} \, \frac{g_k(\omega)}{ \langle k \rangle^{\frac{1}{2} + \theta}} \, e^{i k x} \end{equation} where $g_k$ are i.i.d Gaussian random variables, $ \langle k \rangle:=  \sqrt{1 + |k|^2}$,  and $\theta >0$ is fixed but arbitrary,  the solution to \eqref{dnls} or equivalently \eqref{gdnls} almost surely exists for a positive time.
\subsection{Ideas from probabilistic setting}
In the probabilistic PDE context (i.e. random data theory for dispersive and wave equations or parabolic stochastic PDEs) 
where one deals with randomized initial data or a random forcing term,  the idea of exploiting the structure of the solution has been used for a long time, see for example Bourgain \cite{B5, B6} in the context of the defocusing (Wick ordered) cubic nonlinear Schr\"odinger equation{\footnote{ See also, more recent work by Burq and Tzvetkov in the context of nonlinear wave equations \cite{BTz1}.}}, and Da Prato-Debussche \cite{DPD, DPD2} in the context of the stochastic Navier-Stokes and the stochastic quantization equations. More recently this idea has been exploited in a large body of work by many authors. See for example \cite{B5, B6,  Th, CoOh, BTT, Deng2, NS, HiyOk, DLM, BOP2, BOP, Yue} and references therein for some works on the random data local Cauchy theory in the context of  nonlinear Schr\"odinger equations. The key point is that, if one considers the linear evolution of random data (or random forcing), 
then almost surely, it enjoys much better estimates than arbitrary functions of the same regularity. In turn this allows one to \emph{re-center} the solution around the linear evolution of random data (or around higher order iterates), and conclude that the difference between the two belongs to a Banach space of higher regularity than the one dictated by the (weaker) regularity of the random initial data.

For example, in Bourgain \cite{B5}, which deals with the cubic nonlinear Schr\"{o}dinger equation on $\mathbb{T}^2$, the initial data $\phi^{\omega}$ belongs to Sobolev $H^{-\varepsilon}$ almost surely, for any $\varepsilon>0$, whence its linear evolution only belongs to the Fourier restriction norm spaces $X^{-\varepsilon,\frac{1}{2}+}$ almost surely.  On the other hand, the equation is $L^2$ critical, so if one were to try to prove local well posedness via a fixed point argument, the needed trilinear estimates would fail for arbitrary functions in $X^{-\varepsilon, \frac{1}{2}}$.  Instead, Bourgain \cite{B5} constructed solutions $u$ centered around the random linear evolution $\Psi^{\omega} := e^{i t \Delta} \phi^{\omega}$. That is of the form:
\begin{equation}\label{B-trick}
u =    \Psi^{\omega}  +  R, \quad \quad \text{where} \quad  \quad  (i\partial_t + \Delta )R =  \mathcal N( \Psi^{\omega}  +  R),  \end{equation}
and where we have denoted by $\mathcal N$  the Wick ordered cubic nonlinearity. 
Then, almost surely, the needed trilinear estimates for  $\mathcal N( \Psi^{\omega}  +  R)$ hold true and the solution $R$ to the difference equation in \eqref{B-trick} can be constructed in a smoother space $X^{\varepsilon,\frac{1}{2}+}$ by a contraction mapping argument. Heuristically, one should view \eqref{B-trick} as a `hybrid equation' which on the one hand behaves subcritically in $R$, thus locally well-posed in $H^{\varepsilon}$;  while on the other hand the random linear evolutions $\Psi^{\omega}$ behave better than an arbitrary function in $X^{-\varepsilon,\frac{1}{2}+}$ when they are entries in $\mathcal N( \Psi^{\omega}  +  R)$ thanks to large deviation estimates.  
A similar phenomenon happens in Da Prato-Debussche's argument for-for example-the stochastic Navier-Stokes equation on $\mathbb{T}^2$ with spacetime white noise forcing $\zeta$ \cite{DPD} where the role of $\Psi^\omega$ is replaced by $Z$, the linear evolution of white noise,  $Z_t + \Delta Z = \zeta.$

In both cases, the method can be understood as constructing solutions in a (random affine) submanifold $\mathcal{W}$ consisting of functions belonging to a ball in a smoother space, centered at the random linear evolution.

In the past few years, Gubinelli, Imkeller and Perkowski  \cite{GIP, GIP2}  (see also \cite{CCh} and \cite{GP}) developed a far-reaching generalization of this re-centering method based on the idea of {\it para-controlled distributions}. This is an analytic counterpart to the theory of regularity structures developed by Hairer \cite{Hairer, Hairer1, Hairer2, Hairer3}.  Roughly speaking, in addition to the linear evolution and possibly (suitably renormalized) higher order expressions of the linear evolution, one moves to the new `center'  terms that are `para-controlled' by such expressions. Here a function $f$ is said to be para-controlled by a function $g$ if, up to some smoother `remainder' terms, $f$ can be written as the Bony para-product between high frequencies of $g$ and low frequencies of some auxiliary function $h$, namely that\begin{equation}\label{para}
f = \Pi_{>}(g, h)+R := \sum_N  P_N g \cdot P_{\ll N} h +R, 
\end{equation} where for dyadic frequencies $N$, $P_N$ and $P_{\ll N}$ are the standard Littlewood-Paley operators projecting onto frequencies $\sim N$ and $\ll N$ respectively, 
and $R$ is smoother than $f$. An example is the (parabolic) $\Phi^4_3$ model, which is the cubic heat equation on $\mathbb T^3$ with white noise forcing \cite{GIP, CCh}, where one constructs solutions of the form 
\begin{equation}\label{phi43}
u =  Z + I( P_3(Z))  \, + \,  I \, \Pi_{>}( P_2(Z), u-Z) \, + \, R,
\end{equation} where $Z$ is the linear evolution of the white noise, $P_3(Z)$ consists of other structured components  which themselves are given in term of (suitably renormalized cubic) powers of $Z$, 
$I$ is the Duhamel operator $(\partial_t-\Delta)^{-1}$, and $R$ is a remainder, that has higher regularity. Here then the solution consists of the linear evolution $Z$, a higher order expression $I(P_3(Z))$, a para-controlled part term $I\,\Pi_{>}( Z^2, u-Z)$ and a remainder and thus belongs to a random submanifold which is much more nonlinear.

These ideas have been extensively used in various stochastic contexts in recent years by many authors. We refer the reader for example to work by  Mourrat and Weber \cite{MouWe} and to Mourrat, Weber and Xu \cite{MouWeX} and to references therein for further work in the context of the $\Phi^4_3$ model and to Chandra and Weber \cite{ChWe} and references therein for a nice survey of these ideas. See also \cite{BaiBer}.  We also refer to recent work by Gubinelli, Koch and Oh \cite{GKO, GKO2} where these ideas were applied to the stochastic nonlinear wave equation with quadratic nonlinearity in $\mathbb T^2$ and in $\mathbb T^3$ respectively.

\subsection{The deterministic context of DNLS} Inspired by the ideas in the probabilistic setting described above, in this paper we develop a new \emph{deterministic} method to describe the structure of solutions $v$ to the Cauchy initial value problem for \eqref{dnls} with data at almost critical regularity.  A review of all of the above examples suggests that, if we were in the probabilistic setting (i.e. (\ref{random-data})), we should look for solutions essentially of form
\[u=w+\textrm{(terms para-controlled by $w$)}+\textrm{(smooth remainders)},\] where $w$ is the combination of the random linear evolution, and multilinear expressions dictated by the random linear evolution. The choice of such $w$ is forced upon us (one can at most choose the order of expansion) by the fact that one needs to (and indeed can) gain from the exact Gaussian structure.

In the deterministic setting, there is no gain from randomness. One could try to mimic the probabilistic construction of para-controlled terms in previous works, and arrive at the ansatz 
\begin{equation}\label{ansatz1}
v =w\, +\, I \, \Pi_{>}^{(1)} (\partial_x  \overline{w},v,v  ) \, + \, I \,\Pi_{>}^{(2)} (\partial_x  \overline{w},w,v  ) +  \text{(smooth remainders)}, 
\end{equation} where $I$ is the Duhamel operator 
\[IF(t)=\int_0^te^{i(t-s)\partial_x^2}F(s)\,\mathrm{d}s\] and the cubic para-products are defined by
\begin{equation}\label{ansatz-new}
\Pi_{>}^{(1)} (\partial_x  \overline{w},v,v  )  = \sum_N  P_N \partial_x\overline{w}\cdot ( P_{\ll N} v)^2 ,\quad \Pi_{>}^{(2)} (\partial_x  \overline{w},w,v  ) = \sum_N  P_N \partial_x\overline{w}\cdot  (P_{N} wP_{\ll N}v).
\end{equation} However unlike the probabilistic setting, we are no longer guided by the
Gaussians and need to find the right $w$ ourselves.   The naive choice of linear evolution for $w$ is doomed to fail, and even if one includes multilinear expressions of the linear evolution, calculations show that in the absence of randomness, one would need to expand to a very high (if not infinite) order before unearthing `smooth remainders' that have enough regularity (namely $H_{4-}^{\frac{1}{2}}$ due to \cite{GH}) to close the estimates. With a high order of expansion, the terms involved then quickly become too complex to control in our setting. Work in this direction was considered by the second author together with Chanillo, Czubak, Mendelson and Staffilani in the context of the nonlinear wave equation with quadratic derivative nonlinearities, see \cite{CCMNS} for details.

To get out of this maze, in this paper we will give up the idea of fixing $w$ to be some explicit multilinear expression dictated by the linear evolution. Instead we will construct this $w$, which para-controls the solution $v$, \emph{dynamically}. That is, we take all the linear and higher order terms in the above-mentioned expansion, as well as the presumed smooth error terms, and put them into a single `center' $w$. This leads to the new ansatz
\begin{equation}\label{ansatz-2}
v\, =\, w \, + \,   I \, \Pi_{>}^{(1)} (\partial_x\overline{w},v,v) \,+\,    I \, \Pi_{>}^{(2)} (\partial_x\overline{w},w,v),
\end{equation} where  $w$ is the `center'  which itself moves together with $v$ and belongs to some subspace of $X^{\sigma, \frac{1}{2}+}_{p, 2}, \, p \geq 4$.   As it turns out, uncovering the final structure of $v$ is slightly more complex but \eqref{ansatz-2} conveys the main philosophy{\footnote{ Note that the decomposition of v is nonlinear both in w and in the para-controlled terms.}} (see Section \ref{structure} for details).  Since $w$ does not have a specific multilinear structure, one difficulty is identifying the right space where $w$ will lie. 
By carefully analyzing the terms that are expected to appear in $w$, we can specify this space{\footnote{ More precisely for $\delta >0$ and $q = \frac{1}{4\delta}$, the right space is $X^{\frac{1}{2}, 1-2 \delta}_{p, q}$. See Section 3 for precise definitions.}} to be $X^{\frac{1}{2}, 1-}_{p, \infty-}$.

\smallskip

A final complication comes from the fact that unlike the parabolic setting where the Duhamel operator $I$ automatically gains two derivatives, such gain is not automatic for the Schr\"odinger equation. Rather, it has to be manually induced by performing a frequency cut-off also in the Fourier variable of time, so as to restrict to the region where the parabolic weight in frequency (which is the one that appears in the $X^{s,b}_{p,q}$ norms) is large. In principle this would require that we replace in our ansatz \eqref{ansatz-2} the Duhamel operator $I$ by a frequency cut-off version of it, which would introduce non-locality in time which could be incompatible with local in time solutions. Fortunately the frequency cut-off can be substituted by a suitable time convolution
\begin{equation}\label{time-localiz}
\widetilde I F(t) \,: = \, \int_0^t \, \chi( k^2 (t-s)) e^{i (t-s) \partial_x^2} F(s) ds 
\end{equation} which has the same effect for some carefully chosen $\chi$. 
See Section \ref{defofe} for details. 

\smallskip

With the above discussion, we can now fix the submanifold $\mathcal W$, in which the solution $v$ to \eqref{gdnls} is uniquely constructed, to be
\begin{equation}\label{subm}
\mathcal W = \big\{v \in X^{\sigma, \frac{1}{2}+}_{p, 2}\, :\, v = w +    \widetilde I \, \Pi_{>}^{(1)} (\partial_x\overline{w},v,v)+    \widetilde I \, \Pi_{>}^{(2)} (\partial_x\overline{w},w,v), \quad w \in X^{\frac{1}{2}, 1-}_{p, \infty-} \big\}.
\end{equation} We will show that the submanifold $\mathcal{W}$ is well-defined, parametrized by $w\in X^{\frac{1}{2}, 1-}_{p, \infty-}$, and that the trilinear estimates (\ref{trilinear-ce}), which fail for arbitrary input functions in $X^{\sigma, \frac{1}{2}+}_{p, 2}$, $p \geq 4$ are actually true for input functions in $\mathcal{W}$. These together will allow one to construct the solution $v\in\mathcal{W}$ by a contraction mapping argument. Finally, by inverting the gauge transform, one can construct the solution $u$ to (\ref{dnls}) in $\mathcal{Z}$, which is the preimage of $\mathcal{W}$ under the gauge transform. See Section \ref{structure} for details.

\smallskip 

\begin{rem}
We conclude this introductory discussion by noting that there is a large body of work that has contributed to our current understanding of the Cauchy problem for \eqref{dnls} for data in the Sobolev spaces $H^s, \, s \geq \frac{1}{2}$ both in the periodic and non-periodic settings; we refer the reader to \cite{KaupNew, TF, Hay, HayOz1, HayOz2, Oz, Tak1, Tak2,  BiaLin, CKSTT1, CKSTT2, Herr, Win,  MWX, MoOh, Pelin, Jenkins1, Jenkins2, MoYo} and references therein for a more comprehensive treatment. 
\end{rem}

\medskip 

\subsection{Plan of the paper}
 The paper is organized as follows. In Section \ref{thegauge} we recall the periodic gauge transformation used in \cite{GH} and perform such transformation to \eqref{dnls}. Then, we lay out the set up and frequency interactions splitting of the nonlinearities in the gauged derivative Schr\"odinger equation which will guide our analysis. 
In Section \ref{function} we define and set up our function spaces, prove the main linear estimates and prove an improved divisor bound which is used in some our estimates. In Section \ref{structure} we discuss the structure of the solution, identifying the para-controlling terms and the precise solution submanifold $\mathcal W$ where $v$ will belong.  In Section \ref{recoverv} we prove a prior bounds for the para-controlling terms. Section \ref{recoverw} constitutes the heart of the paper. Here we find $w$ and prove all the underlying multilinear estimates involved in its construction. In the course of the proof we show in particular that all relevant nonlinearities are well defined as space-time distributions whence the integral equation \eqref{intg} for $u$ will be equivalent to the integral equation formulation of \eqref{gdnls} (see Section \ref{setup} for details). 
 Finally in Section \ref{preg} we prove a preservation of regularity result.

\medskip

\subsection{Notations and parameters}\label{not} We will use the notation
\[\mathbb{P}_0h=\frac{1}{2\pi}\int_{\mathbb{T}}h\,\mathrm{d}x,\quad \mathbb{P}_{\neq 0}h=h-\mathbb{P}_0h.
\] The space, time and spacetime Fourier transforms are respectively defined as \[\widehat{u}(k)=\mathscr{F}_xu(k)=\frac{1}{2\pi}\int_{\mathbb{T}}e^{-ikx}u(x)\,\mathrm{d}x,\quad \widehat{u}(\xi)=\mathscr{F}_tu(\xi)=\frac{1}{2\pi}\int_{\mathbb{R}}e^{-i\xi t}u(t)\,\mathrm{d}t,\]\[\widehat{u}(k,\xi)=\mathscr{F}u(k,\xi)=\frac{1}{(2\pi)^2}\int_{\mathbb{R}\times\mathbb{T}}e^{-i(kx+\xi t)}u(t,x)\,\mathrm{d}t\mathrm{d}x.\] so $\mathscr{F}$ is reserved for the spacetime Fourier transform. As for $\widehat{u}$, whether it means space, time or spacetime Fourier transform will be clear from the context. The integral over the set \[\{(\lambda_1,\cdots,\lambda_r):\lambda_1\pm\cdots\pm\lambda_r=\mu\}\] for fixed $\mu$ will be with respect to the Lebesgue measure $\mathrm{d}\lambda_1\cdots\mathrm{d}\lambda_{r-1}$.  
We denote by $\mathbf{1}_{P}$ the characteristic function of a set or property $P$.
\smallskip

Recall that $p_0$ is fixed; we will fix a small parameter $0<\delta\ll 1$ depending on $p_0$, and define the other parameters $(b_0,b_1,q_0,q_1,r_0,r_1,r_2)$ as follows: \begin{equation}\label{param}b_0=1-2\delta,\quad b_1=1-\delta,\quad q_0=\frac{1}{4\delta},\quad q_1=\frac{1}{(4.5)\delta},\quad \frac{1}{r_0}=\frac{1}{2}+\delta,\quad \frac{1}{r_1}=\frac{1}{2}+2\delta,\quad \frac{1}{r_2}=\frac{1}{2}+3\delta.\end{equation} We also use $\theta$ to denote a generic positive quantity that is sufficiently small depending on $\delta$ (so $\theta$ may have different values at different instances.) 
\smallskip

We will fix $A$ as in the statement of Theorem \ref{main}, and let $A_1$ be large depending on $A$, $A_2$ be large depending on $A_1$, etc. All implicit constants below will depend on these $A_j$'s and the above parameters. The time length $T$ will also be fixed, and small enough depending on these implicit constants.

\section{The gauge transform and other reductions}\label{thegauge}
\subsection{The gauge transform} Notice that $\mathbb{P}_0|u|^2$ is conserved under the flow of (\ref{dnls}). Consider the gauge transform, see \cite{GH},
\begin{equation}
\label{gauge}v(t,x)=(\mathcal{G}u)(t,x):=(\mathcal{G}_0u)\big(t,x-2\mathbb{P}_0|u|^2 \, t\big),\quad (\mathcal{G}_0u)(t,x):=e^{-iG(t,x)}\cdot u(t,x),
\end{equation} where \begin{equation}\label{defiu}G=\partial_{x}^{-1}\mathbb{P}_{\neq 0}(|u|^2)
\end{equation} is the unique mean-zero antiderivative of $\mathbb{P}_{\neq 0}|u|^2$. This gauge transform is easily inverted, with inverse given by \begin{equation}\label{invert1}u(t,x)=(\mathcal{G}^{-1}v)(t,x)=e^{iG(t,x)}\cdot v_0(t,x),\quad G=\partial_x^{-1}\mathbb{P}_{\neq 0}|v_0|^2,
\end{equation} where 
\begin{equation}\label{invert2}v_0(t,x)=v(t,x+2\mathbb{P}_0|v|^2\, t).
\end{equation}
\begin{prop}\label{transform} The maps $\mathcal{G}$ and $\mathcal{G}^{-1}$ are continuous from $C_t^0H_{p_0}^{\frac{1}{2}}(J)$ to itself for any interval $J$, and map bounded sets to bounded sets.
\end{prop}
\begin{proof} Notice that $\mathcal{G}=\mathcal{G}_1\mathcal{G}_0$, where
\[\mathcal{G}_0u=\exp(-i\partial_x^{-1}\mathbb{P}_{\neq 0}|u|^2)\cdot u,\quad \mathcal{G}_0^{-1}u=\exp(i\partial_x^{-1}\mathbb{P}_{\neq 0}|u|^2)\cdot u,\] and 
\[\mathcal{G}_1u(t,x)=u(t,x-2\mathbb{P}_0|u|^2\, t),\quad \mathcal{G}_1^{-1}u(t,x)=u(t,x+2\mathbb{P}_0|u|^2\, t).\] In \cite{GH}, Lemma 6.2 and Lemma 6.3, it is proved that $\mathcal{G}_0:H_{p_0}^{\frac{1}{2}}\to H_{p_0}^{\frac{1}{2}}$ is locally bi-Lipschitz, and that $\mathcal{G}_1:C_t^0H_{p_0}^{\frac{1}{2}}(J)\to C_t^0H_{p_0}^{\frac{1}{2}}(J)$ is a homeomorphism. Moreover, it is easily checked that \[\|\mathcal{G}_1u\|_{C_t^0H_{p_0}^{\frac{1}{2}}(J)}=\|\mathcal{G}_1^{-1}u\|_{C_t^0H_{p_0}^{\frac{1}{2}}(J)}=\|u\|_{C_t^0H_{p_0}^{\frac{1}{2}}(J)},\] so $\mathcal{G}$ and $\mathcal{G}^{-1}$ map bounded sets to bounded sets.
\end{proof}
\subsection{The transformed equation}
We calculate that $v=\mathcal{G}u$ satisfies the equation
\begin{equation}\label{dnls2}
(\partial_t-i\partial_x^2)v=\mathcal{C}(v,v,v)+\mathcal{Q}(v,\cdots,v),
\end{equation} where the cubic and quintic nonlinearities are defined as
\begin{equation}\label{cub}\mathscr{F}_x\mathcal{C}(v_1,v_2,v_3)(k)=\sum_{\mathbb{V}_3}k_1M_3(k,k_1,k_2,k_3)\cdot\overline{\widehat{v_1}(k_1)}\widehat{v_2}(k_2)\widehat{v_3}(k_3),
\end{equation} and 
\begin{equation}\label{quin}\mathscr{F}_x\mathcal{Q}(v_1,\cdots,v_5)(k)=\sum_{\mathbb{V}_5}M_5(k,k_1,\cdots,k_5)\cdot\widehat{v_1}(k_1)\overline{\widehat{v_2}(k_2)}\widehat{v_3}(k_3)\overline{\widehat{v_4}(k_4)}\widehat{v_5}(k_5).
\end{equation} The sets $\mathbb{V}_3$ and $\mathbb{V}_5$ are defined by
\begin{equation}\label{sets}\begin{split}\mathbb{V}_3&=\big\{(k_1,k_2,k_3)\in\mathbb{Z}^3:k_2+k_3-k_1=k,|k_2|\geq |k_3|,k\not\in\{k_2,k_3\}\big\}\cup\{(k,k,k)\},\\
\mathbb{V}_5&=\big\{(k_1,\cdots,k_5)\in\mathbb{Z}^5:k_1-k_2+k_3-k_4+k_5=k\big\},
\end{split}\end{equation} and the coefficients $M_j$ are explicitly defined functions, with $|M_j|\lesssim 1$ for $j\in\{3,5\}$. They also have the right symmetry so that (\ref{dnls2}) conserves $\mathbb{P}_0|v|^2$. See \cite{GH} for the precise formulas.

\smallskip

\begin{rem} For integers $k$, $k_1$, $k_2$ and $k_3$ such that $k_2+k_3-k_1=k$, we will rely throughout the proofs on the quantity  $\Delta :=k^2+k_1^2-k_2^2-k_3^2$.
\end{rem}

\subsection{Splitting the cubic nonlinearity} We will further split the cubic nonlinearity $\mathcal{C}$ into four parts: a ``high-high'' part\footnote{Here ``high'' and ``low'' are with respect to the frequencies $k_2$ and $k_3$.}, a ``low-low'' part, a ``semilinear'' part and a ``non-resonant'' part. Decompose $\mathbb{V}_3$ into four subsets:
\begin{equation}\label{subsets}
\begin{aligned}
\mathbb{X}_{H}&=\big\{(k_1,k_2,k_3)\in \mathbb{V}_3:|k_3|\geq 2^{-20}|k|\big\},\\
\mathbb{X}_{L}&=\big\{(k_1,k_2,k_3)\in \mathbb{V}_3:|k_2|<2^{-20}|k|\big\},\\
\mathbb{X}_{S}&=\big\{(k_1,k_2,k_3)\in \mathbb{V}_3:2^{-10}|k_1|\leq |k_3|<2^{-20}|k|\big\},\\
\mathbb{X}_{N}&=\mathbb{V}_3-(\mathbb{X}_{H}\cup\mathbb{X}_{L}\cup \mathbb{X}_S).
\end{aligned}
\end{equation} The following properties of this splitting are elementary and so we omit the proof.
\begin{prop}\label{splitting} We have the following properties for the sets $\mathbb{X}_*$ where $*\in\{H,L,S,N\}$:
\begin{enumerate}
\item For $(k_1,k_2,k_3)\in \mathbb{X}_{H}$ we have $|k_2|\geq|k_3|\geq 2^{-20}|k|$.

\item For $(k_1,k_2,k_3)\in \mathbb{X}_{L}$ we have $|k|/2\leq|k_1|\leq 2|k|$ and $\min(|k|,|k_1|)\geq 2^{18}\max(|k_2|,|k_3|)$.

\item For $(k_1,k_2,k_3)\in \mathbb{X}_{S}$ we have $|k|/2\leq |k_2|\leq 2|k|$, $|k|\geq 2^{20}|k_3|$ and $|k_3|\geq 2^{-10}|k_1|$.

\item For $(k_1,k_2,k_3)\in\mathbb{X}_N$ we have $|k_2|\geq 2^{-22}\max(|k|,|k_1|)$ and $\min(|k|,|k_1|)\geq 2^{10} |k_3|$.

\item For $(k_1,k_2,k_3)\in \mathbb{X}_{H}\cup\mathbb{H}_{S}$ we have  \begin{equation}\label{semilinear}|k_1|\cdot(\langle k_1\rangle\langle k_2\rangle\langle k_3\rangle)^{-\frac{1}{2}}\lesssim\langle k\rangle^{-\frac{1}{2}}.\end{equation}

\item For $(k_1,k_2,k_3)\in \mathbb{X}_{L}\cup\mathbb{X}_{N}$ we have
\begin{equation}\label{nonres}|\Delta|\sim \langle k\rangle\langle k_1\rangle,\quad\textrm{where}\quad\Delta=k^2+k_1^2-k_2^2-k_3^2\,=\, 2(k-k_2)(k -k_3).\end{equation}
\end{enumerate}
\end{prop}

\smallskip

We also need the following result, which will be used in analyzing the quintic terms in Section \ref{recoverw}. Once again these properties are elementary. We omit the proof.

 \begin{prop}\label{quinprep} 
Suppose\[k_2+k'-k_1=k,\,\,(k_1,k_2,k')\in\mathbb{X}_{\ast}; \quad k_4+k_5-k_3=k',\,\,(k_3,k_4,k_5)\in\mathbb{X}_{\#},\] where $*,\, \# \in\{H,L,S,N\}$ and where here
\begin{equation*}
\Delta\,=\, k^2 + k_1^2 - k'^2 - k_2^2 \qquad \text{ and } \qquad \Delta'\, =\,  k'^2 + k_3^2 - k_5^2 - k_4^2.
\end{equation*}

Moreover let us define, \[\alpha:=\frac{|k_1||k_3|}{\langle\Delta\rangle},\quad \beta:=\frac{|k_1||k_3|}{\langle\Delta\rangle\langle\Delta'\rangle},\quad \gamma:=\frac{|k_1||k_3|}{\langle\Delta'\rangle}.\] Then we have the followings:
\begin{enumerate}
\item Assume $*\in\{H,S\}$ and  $\#\in\{L,N\}$.  Then  either $(1a)$\,  $|\gamma|\lesssim 1$  or  $(1b)$\, $*=H$ and $|k_1|\geq 2^{40}|k'|$.  In case $(1b)$, 
\smallskip
\begin{itemize}
\item[(i)] if $\#=L$,  or if $\#=N$ and $|k_3|\leq 2^{30}|k'|$, then we have that, 
\[|k_1|/2\leq |k_2|\leq 2|k_1|,\quad |k_1|\geq 2^{5}\max_{3\leq j\leq 5}|k_j|,\quad |\gamma|\lesssim\frac{\langle k_1\rangle}{\max(\langle k_3\rangle,\langle k_4\rangle,\langle k_5\rangle,\langle k\rangle)},\] that $k_1\neq k_2$, and that $\max(|k_3|,|k_4|,|k_5|)=|k_j|$ for $j=3$ if $\#=L$ and,  $j\in\{3,4\}$ if $\#=N$;
\smallskip
\item[(ii)]  if $\#=N$ and $|k_3|\geq 2^{30}|k'|$, then we have that
\begin{equation*}
\begin{aligned}|k_1|/2&\leq |k_2|\leq 2|k_1|,&|k_3|/2&\leq |k_4|\leq 2|k_3|,\\|k_1|&\geq 2^5\max(|k|,|k_5|),&|\alpha|&\lesssim\frac{\langle k_1\rangle}{\max(\langle k\rangle,\langle k_5\rangle,\langle\pm k_1\pm k_2\rangle)}.
\end{aligned}
\end{equation*}
\end{itemize}
  
\item Assume $*,\#\in\{L,N\}$, then either $(2a)$\, $|\alpha|\lesssim 1$ or $(2b)$\, $\#=N$ and $|k_3|\geq 2^{40}|k|$. In case $(2b)$, 
\smallskip
\begin{itemize}
 \item[(i)] if $*=L$ or if $*=N$ and $|k_1|\leq 2^{30}|k|$, then we have
\[|k_3|/2\leq |k_4|\leq 2|k_3|,\quad |k_3|\geq 2^{5}\max_{j\in\{1,2,5\}}|k_j|,\quad |\gamma|\lesssim\frac{\langle k_3\rangle}{\max(\langle k_1\rangle,\langle k_2\rangle,\langle k_5\rangle,\langle k\rangle)},\] that $k_3\neq k_4$, and that $\max(|k_1|,|k_2|,|k_5|)= |k_j|$ for $j=1$ if $*=L$ and,  $j\in\{1,2\}$ if $*=N$; 
\smallskip
\item[(ii)] if $*=N$ and $|k_1|\geq 2^{30}|k|$, then we have that
\begin{equation*}
\begin{aligned}|k_1|/2&\leq |k_2|\leq 2|k_1|,&|k_3|/2&\leq |k_4|\leq 2|k_3|,\\|k_3|&\geq 2^5\max(|k|,|k_5|),&|\alpha|&\lesssim\frac{\langle k_3\rangle}{\max(\langle k\rangle,\langle k_5\rangle)},&|k|&\neq |k_5|.
\end{aligned}
\end{equation*}
\end{itemize}

\item Assume $*,\#\in\{L,N\}$, then we have that
\[|\beta|\lesssim\frac{1}{\langle k\rangle\langle \pm k_3\pm k_4\rangle},\quad \langle k\rangle\gtrsim\langle k_5\rangle.\]
\end{enumerate}
\end{prop}

\medskip

For $*\in\{H,L,S,N\}$ define $\mathcal{C}_*$ by
\begin{equation}\label{cubpart}\mathscr{F}_x\mathcal{C}_*(v_1,v_2,v_3)(k)=\sum_{\mathbb{X}_*}k_1M_3(k,k_1,k_2,k_3)\cdot\overline{\widehat{v_1}(k_1)}\widehat{v_2}(k_2)\widehat{v_3}(k_3),
\end{equation}then we have
\begin{equation}\label{decomposition}
\mathcal{C}=\mathcal{C}_H+\mathcal{C}_{L}+\mathcal{C}_{S}+\mathcal{C}_N.
\end{equation}
\subsection{The full setup}\label{setup} When all the relevant nonlinearities are well-defined as spacetime distributions, which we will see in the course of the proof, the integral equation (\ref{intg}) for $u$ will be equivalent to the integral equation
\begin{equation}\label{intg2}
v(t)=e^{it\partial_x^2}v_0+I(\mathcal{C}(v,v,v)+\mathcal{Q}(v,\cdots,v)),\quad IF(t)=\int_0^t e^{i(t-s)\partial_x^2}F(s)\,\mathrm{d}s,
\end{equation} for $v$, where the nonlinearities $\mathcal{C}$ and $\mathcal{Q}$ are as in (\ref{cub})$\sim$(\ref{sets}), and the initial data
\begin{equation}\label{lin3}v_0=\exp(-i\partial_x^{-1}\mathbb{P}_{\neq 0}|u_0|^2)\cdot u_0,
\end{equation} which satisfies $\|v_0\|_{H_{p_0}^{\frac{1}{2}}}\leq A_1$ given that $\|u_0\|_{H_{p_0}^{\frac{1}{2}}}\leq A$.

In the proof we will be extending the function $v$, which is defined on $J=[-T,T]$, to the whole line $\mathbb{R}_t$; to this end we fix a smooth function $\varphi(t)$ that is $1$ for $|t|\leq 1$ and $0$ for $|t|\geq 2$, and define the truncated versions of the linear solution and Duhamel operator \begin{equation}\label{duhamel}\psi(t)=\varphi(t)\cdot e^{it\partial_x^2}v_0,\quad \mathcal{I}F(t)=\varphi(t)\cdot I(\varphi(s)\cdot F(s)).
\end{equation} For later uses we will also define $\varphi_T(t)=\varphi(T^{-1}t)$.
\section{Preparations}\label{function}

In this section we define and set up our function spaces, prove the main linear estimates and prove an improved divisor bound which is used in some our estimates. 

\subsection{Function spaces}
We begin by properly defining the functions spaces that play a role in our proof. Denote{\footnote{Note that in \cite{GH} this same space is denoted by $\widehat{H}^s_{p'}(\mathbb T)$ where $\frac{1}{p}+ \frac{1}{p'} = 1$.}}  the Fourier-Lebesgue norms $H_p^s(\mathbb T)$, where $p\in[2,\infty)$ an $s\in\mathbb{R}$, by
\begin{equation}\label{defhp}\|u\|_{H_p^{s}}=\|\langle k\rangle^{s}\widehat{u}(k)\|_{\ell_k^p}.\end{equation} 
In one dimension,  $H_p^s$ has the same scaling as the Sobolev space $H^\gamma$ for \begin{equation}\label{scale}\gamma= s+\frac{1}{p}-\frac{1}{2} \end{equation} When $p>2$ we have $\gamma< s$, which allows the regularity index $\gamma$ to decrease while keeping $s \geq\frac{1}{2}$.

\noindent The associated Fourier restriction norm spaces $X_{p,q}^{s,b}$, where $p,q\in[2,\infty)$ an $s,b\in\mathbb{R}$, are then defined by
\begin{equation}\label{defxpq}\|u\|_{X_{p,q}^{s,b}}=\|\langle k\rangle^s\langle\xi+k^2\rangle^b\widehat{u}(k,\xi)\|_{\ell_k^pL_{\xi}^q}.\end{equation} 
\smallskip

For $2 \leq p_0 < \infty$ fixed and $0<\delta\ll 1$ small depending on $p_0$ also fixed,  let the parameters $(b_0,b_1,q_0,q_1,r_0,r_1,r_2)$ be defined as in \eqref{param}.  We define the four spaces in which the estimates are proved as follows:
\begin{equation}\label{defspaces}
\begin{aligned}
Y_0&\,=\,X_{p_0,r_0}^{\frac{1}{2},\frac{1}{2}},\qquad &Y_1&\, =\, X_{p_0,r_1}^{\frac{1}{2},\frac{1}{2}},\\
Z_0&\, =\, X_{p_0,q_0}^{\frac{1}{2},b_0},\qquad &Z_1&\, =\, X_{p_0,q_0}^{\frac{1}{2},b_1}.
\end{aligned}
\end{equation}
Note that by H\"{o}lder we have
\begin{equation}\label{embedding}X_{p,q}^{s,b}\subset X_{p',q'}^{s',b'},\quad\textrm{provided}\quad p\leq p',q\leq q';\,\,s+\frac{1}{p}<s'+\frac{1}{p'},\,\,b+\frac{1}{q}<b'+\frac{1}{q'},\end{equation} in particular $Z_0\subset Y_0\subset C_t^0H_{p_0}^{\frac{1}{2}}$.
Finally, for any finite interval $I$ and any spacetime norm $\mathcal{Y}$, define \begin{equation}\label{defz}\|u\|_{\mathcal{Y}(I)}=\inf\big\{\|v\|_{\mathcal{Y}}:v=u\mathrm{\ on\ }I\big\}.\end{equation}
\subsection{Linear estimates} We will be using the following notation for a spacetime function $F$:
\begin{equation}\label{tilted}\widetilde{F}(k,\lambda)=\mathscr{X}F(k,\lambda):=\widehat{F}(k,\lambda-k^2),
\end{equation} where $\widehat{F}$ is the spacetime Fourier transform.
\begin{lem}\label{lin} Define the function
\begin{equation}\label{linfor1}K(\lambda,\sigma)=i\bigg[\int_{\mathbb{R}}\frac{\widehat{\varphi}(\lambda-\mu)\widehat{\varphi}(\mu-\sigma)}{\mu}\,\mathrm{d}\mu-\widehat{\varphi}(\lambda)\int_{\mathbb{R}}\frac{\widehat{\varphi}(\mu-\sigma)}{\mu}\,\mathrm{d}\mu\bigg],
\end{equation}
where integrations are defined as principal value limits, then it satisfies \begin{equation}\label{esti}|K(\lambda,\sigma)|\lesssim_B\bigg(\frac{1}{\langle\lambda\rangle^B}+\frac{1}{\langle\lambda-\sigma\rangle^B}\bigg)\frac{1}{\langle \sigma\rangle}
\end{equation} for any $B>0$, and we have\begin{equation}\label{linfor2}
\widetilde{\mathcal{I}F}(k,\lambda)=\int_{\mathbb{R}}K(\lambda,\sigma)\widetilde{F}(k,\sigma)\,\mathrm{d}\sigma.
\end{equation} 
\end{lem}
\begin{proof} In \cite{Deng}, Lemma 3.3, it is derived that
\[\widetilde{\mathcal{I}F}(k,\lambda)=c_0\int_{\mathbb{R}}\frac{\widehat{\varphi}(\lambda-\mu)}{\mu}\,\mathrm{d}\mu\int_{\mathbb{R}}\widehat{\varphi}(\mu-\sigma)\widetilde{F}(k,\sigma)\,\mathrm{d}\sigma+c_1\widehat{\varphi}(\lambda)\cdot\int_{\mathbb{R}}\frac{\mathrm{d}\mu}{\mu}\int_{\mathbb{R}}\widehat{\varphi}(\mu-\sigma)\widetilde{F}(k,\sigma)\,\mathrm{d}\sigma,\] where $c_0$ and $c_1$ are numerical constants, and integrations are defined as principal value limits. By our convention with Fourier transform, we can calculate that $c_0=i$ and $c_1=-i$, which gives the formula (\ref{linfor1}). The bound (\ref{esti}) follows easily, using that $\widehat{\varphi}$ is a Schwartz function.
\end{proof}
\begin{prop}\label{stcut} Suppose $u$ is a smooth function such that $u(0)=0$. Then we have the estimates
\begin{equation}\label{shorttime}\|\varphi_T\cdot u\|_{Y_0}\lesssim T^{\theta}\|u\|_{Y_1},\quad \|\varphi_T\cdot u\|_{Z_0}\lesssim T^{\theta}\|u\|_{Z_1}.
\end{equation}
\end{prop}
\begin{proof} First notice that, by (\ref{embedding}),
\[\|u\|_{X_{p_0,q_1}^{\frac{1}{2},b_0}}\lesssim\|u\|_{X_{p_0,q_0}^{\frac{1}{2},b_1}}=\|u\|_{Z_1}.\] Then, by separating different Fourier modes and conjugating by $e^{\pm itk^2}$ at Fourier mode $e^{ikx}$, it suffices to prove that for any function $g=g(t)$ satisfying $g(0)=0$,
\begin{equation}\label{st2}\|\langle \lambda\rangle^{b}(\widehat{g}*\widehat{\varphi_T})(\lambda)\|_{L^q}\lesssim T^{\frac{1}{\widetilde{q}}-\frac{1}{q}} \, \|\langle \sigma\rangle^b\widehat{g}(\sigma)\|_{L^{\widetilde{q}}},
\end{equation} provided $\infty>q>\widetilde{q}>1$ and $b+ \frac{1}{\widetilde{q}} >1>b$. Let $g=g_1+g_2$ where
\[\widehat{g_1}(\sigma)=\mathbf{1}_{|\sigma|\geq T^{-1}}(\sigma)\widehat{g}(\sigma),\quad \widehat{g_2}(\sigma)=\mathbf{1}_{|\sigma|< T^{-1}}(\sigma)\widehat{g}(\sigma),\]we will actually prove that
\begin{equation}\label{st21}\|\langle\lambda\rangle^{b}(\widehat{g_1}*\widehat{\varphi_T})(\lambda)\|_{L^q}\lesssim T^{\frac{1}{\widetilde{q}}-\frac{1}{q}} \, \|\langle \sigma\rangle^b\widehat{g}(\sigma)\|_{L^{\widetilde{q}}},
\end{equation} and
\begin{equation}\label{st22}\|\langle\lambda\rangle^{b}(\widehat{g_2}*\widehat{\varphi_T})(\lambda)\|_{L^q}\lesssim T^{\frac{1}{\widetilde{q}}-\frac{1}{q}} \,\|\langle \sigma\rangle^b\widehat{g}(\sigma)\|_{L^{\widetilde{q}}}.
\end{equation} 
To prove (\ref{st21}), we can reduce it to the $L^{\widetilde{q}}\to L^q$ bound for the operator
\[\widehat{g}(\sigma)\mapsto\int_{\mathbb{R}}R(\lambda,\sigma)\widehat{g}(\sigma)\,\mathrm{d}\sigma,\quad R(\lambda,\sigma)=\mathbf{1}_{|\sigma|\geq T^{-1}}\cdot T\widehat{\varphi}(T(\lambda-\sigma))\frac{\langle\lambda\rangle^{b}}{\langle \sigma\rangle^b}.\] Since 
\[\mathbf{1}_{|\sigma|\geq T^{-1}}\cdot\frac{\langle\lambda\rangle^{b}}{\langle \sigma\rangle^b}\lesssim\frac{\langle T\lambda\rangle^b}{\langle T\sigma\rangle^b}\lesssim\langle T(\lambda-\sigma)\rangle^b,\] it follows from Schur's estimate that this $L^{\widetilde{q}}\to L^q$ bound is at most\[\|T\widehat{\varphi}(T\mu)\langle T\mu\rangle^b\|_{L_\mu^\beta}\lesssim T^{1- \frac{1}{\beta}}=
T^{\frac{1}{\widetilde{q}}-\frac{1}{q}},\quad \frac{1}{\beta}=1+\frac{1}{q}-\frac{1}{\widetilde{q}},\] which proves (\ref{st21}).

To prove (\ref{st22}), notice that
\[(\widehat{g_2}*\widehat{\varphi_T})(\lambda)=-T\widehat{\varphi}(T\lambda)\int_{|\sigma|\geq T^{-1}}\widehat{g}(\sigma)\,\mathrm{d}\sigma-\int_{|\sigma|<T^{-1}}T\widehat{g}(\sigma)\big[\widehat{\varphi}(T\lambda)-\widehat{\varphi}(T(\lambda-\sigma))\big]\,\mathrm{d}\sigma\] since $g(0)=0$; thus 
\[|(\widehat{g_2}*\widehat{\varphi_T})(\lambda)|\lesssim_B T\langle T\lambda\rangle^{-B}\int_{\mathbb{R}}\min(1,|T\sigma|)|\widehat{g}(\sigma)|\,\mathrm{d}\sigma\] for any $B>0$. Since by elementary calculation we can prove
\[\int_{\mathbb{R}}\min(1,|T\sigma|)|\widehat{g}(\sigma)|\,\mathrm{d}\sigma\lesssim \|\langle \sigma\rangle^b\widehat{g}(\sigma)\|_{L^{\widetilde{q}}}\cdot\|\min(1,|T\sigma|)\langle\sigma\rangle^{-b}\|_{L^\gamma}\lesssim T^{b- \frac{1}{\gamma}}=T^{b+ \frac{1}{\widetilde{q}}-1},\quad \frac{1}{\gamma}=1-\frac{1}{\widetilde{q}},\] and that
\[\|T\langle T\lambda\rangle^{-B}\langle\lambda\rangle^{b}\|_{L^q}\lesssim T^{1-b- \frac{1}{q}},\] we deduce (\ref{st22}).
\end{proof}
\begin{rem} The requirement $u(0)=0$ is necessary. Below (\ref{shorttime}) will be applied only for those $u$ of form
\[u(t)=\int_0^t(\textrm{expression}),\] namely $u=\mathcal{I}(\cdots)$ or $u=\mathcal{E}_*^*(\cdots)$, see Section \ref{defofe} for the definition of $\mathcal{E}$, so $u(0)=0$ will always be true.
\end{rem}
\subsection{A divisor bound} Finally,  in this subsection we prove and record an improved divisor bound that will be handy later on in some parts of the proof\footnote{We note that a weaker version of Lemma \ref{newdiv} here already appeared in \cite{GH}, Lemma 3.1.}.
\begin{lem}\label{newdiv} (1) Let $\mathcal{R}=\mathbb{Z}$ or $\mathbb{Z}[\omega]$, where $\omega=\exp(2\pi i/3)$, and fix $\varepsilon>0$. Let $k,q\in\mathcal{R}$ and $\rho>0$ be such that $|q|\geq |k|^{\varepsilon}>0$. Then the number of divisors $r\in\mathcal{R}$ of $k$ that satisfies $|r-q|\leq\rho$ is at most $O_\varepsilon(\rho^\varepsilon)$.

(2) Consider the system
\begin{equation}\label{2sys}
\left\{\begin{split}
\pm a\pm b\pm c&=\rm{const.}\\
\mp a^2\mp b^2\mp c^2&=\rm{const.}
\end{split}\right.
\end{equation} where the signs are arbitrary, but the signs of $\pm a$ and $\mp a^2$ etc. are always the opposite. Assume also that there is no \emph{pairing}, where a pairing means that (say) $a=b$ and the signs of $a$ and $b$ in (\ref{2sys}) are the opposite. Then the number of solutions that satisfy $|a|\sim N_1$, $|b|\sim N_2$ and $|c|\sim N_3$ is $\lesssim_\varepsilon N^\varepsilon$, where $N$ is the \emph{second} largest of the $N_j$'s.
\end{lem}
\begin{proof} (1) It is well-known that $\mathcal{R}$ has unique factorization and satisfies the standard divisor bound: the number of divisors of $k\neq 0$ is at most $O_\varepsilon(|k|^\varepsilon)$. Thus the result is trivial if $\rho\geq |k|^{\delta}$, where $\delta=\varepsilon^4$. Now suppose $|\rho|\leq|k|^{\delta}$ (and $|k|$ is large enough), we claim that the number of divisors $r$ is at most $m-1$, where $m\sim\varepsilon^{-2}$ is an integer.

In fact, suppose $d_j$, where $1\leq j\leq m$ are distinct divisors, then by unique factorization we know that $k$ is divisible by $\mathrm{lcm}(d_1,\cdots,d_m)$, and hence divisible by
\[\frac{\prod_{j=1}^md_j}{\prod_{1\leq i<j\leq m}\gcd(d_i,d_j)}:=k'.\] On the other hand, since $\gcd(d_i,d_j)$ divides $d_i-d_j$, we know $|\gcd(d_i,d_j)|\leq |d_i-d_j|\leq 2\rho$. As also $|d_j|\gtrsim|k|^{\varepsilon}$, we conclude that
\[|k|\gtrsim |k|^{\varepsilon m}|\rho|^{-m^2}\geq |k|^{\varepsilon m-\delta m^2},\] which is impossible for sufficiently large $|k|$, due to our choices of $m$ and $\delta$.

(2) By symmetry we may assume $N':=\max(N_2,N_3)\leq N_1^{1/10}$ (the other case will follow from the same argument below, using standard divisor bounds). There are three cases to consider: when $b+c-a=\rm{const.}$ (and $b^2+c^2-a^2=\rm{const.}$), when $a+b-c=\rm{const.}$, and when $a+b+c=\rm{const.}$.

(a) Suppose $b+c-a=\ell$ is constant and $b^2+c^2-a^2$ is also constant. Then $2(\ell-b)(\ell-c)=\ell^2-(b^2+c^2-a^2):=\Delta$ is also constant, and $\Delta\neq 0$ as their is no pairing. Now choosing $\mathcal{R}=\mathbb{Z}$, $k=\Delta/2$, $q=\ell$ and $\rho\sim N'$ in part (1) yields the result.

(b) Suppose $a+b-c=\ell$ and $a^2+b^2-c^2$ is constant. Then similarly $2(b-c)(\ell-b)=\Delta$ is a nonzero constant. Considering the divisor $\ell-b$ and choosing $\mathcal{R}=\mathbb{Z}$, $k=\Delta/2$, $q=\ell$ and $\rho\sim N'$ in part (1) yields the result.

(c) Suppose $a+b+c=\ell$ and $a^2+b^2+c^2$ is constant. Then letting $u=3a-\ell$ and $v=3b-\ell$, we have that
\[(u-\omega v)(u-\omega^2v)=u^2+uv+v^2=9(a^2+b^2+c^2)-3\ell^2:=\Delta\] is constant. Considering the divisor $u-\omega v$ and choosing $\mathcal{R}=\mathbb{Z}[\omega]$, $k=\Delta$, $q=(\omega+2)\ell$ and $\rho\sim N'$ in part (1) yields the result.
\end{proof}
\section{Structure of the solution}\label{structure} In this section we discuss the structure of the solution, identifying the para-controlling terms and the precise solution submanifold $\mathcal W$ where $v$ will belong.  From now on we will focus on the equation (\ref{intg2}). The submanifold $\mathcal{Z}$ in Theorem \ref{main} will be defined as $\mathcal{Z}=\mathcal{G}^{-1}\mathcal{W}$, where $\mathcal{G}$ is the gauge transform (\ref{gauge}), and $\mathcal{W}$ is a submanifold of $Y_{0}(J)\subset C_t^0H_{p_0}^{\frac{1}{2}}(J)$, in which the solution solution $v$ of (\ref{intg2}) will be constructed. To define $\mathcal{W}$ we need some further preparations.

\subsection{Splitting the Duhamel operator}\label{defofe} Let $\eta(t)$ be a Schwartz function that satisfies the cancellation condition
\begin{equation}\label{cancel}\widehat{\eta}(1)=0,\quad \mathcal{H}\widehat{\eta}(1)=1,
\end{equation} where $\mathcal{H}$ is the Hilbert transform (principal value convolution by $1/\xi$). For $*\in\{N,L\}$, consider the trilinear operator $E_*:=I\mathcal{C}_*$. Recall that $E_*$ satisfies that 
\begin{equation}\label{spacetimec}\mathscr{F}_xE_*(v_1,v_2,v_3)(k,t)=\sum_{\mathbb{X}_*}k_1 M_3(k,k_1,k_2,k_3)\int_0^t e^{-i(t-s)k^2}\overline{\widehat{v_1}(s,k_1)}\widehat{v_2}(s,k_2)\widehat{v_3}(s,k_3)\,\mathrm{d}s.
\end{equation}As before let $\Delta=k^2+k_1^2-k_2^2-k_3^2$ (we always have $|\Delta|\geq 1$), we will define the modified trilinear operators $E_*^Y$ and $E_*^X$ by
\begin{equation}\label{modifiedy}\mathscr{F}_xE_*^Y(v_1,v_2,v_3)(k,t)=\sum_{\mathbb{X}_*}k_1 M_3(k,k_1,k_2,k_3)\int_0^t e^{-i(t-s)k^2}\widehat{\eta}(\Delta(t-s))\overline{\widehat{v_1}(s,k_1)}\widehat{v_2}(s,k_2)\widehat{v_3}(s,k_3)\,\mathrm{d}s
\end{equation} ($Y$ indicated this term is to be estimated in the $Y$ space) and 
\begin{equation}\label{modifiedx}\mathscr{F}_xE_*^X(v_1,v_2,v_3)(k,t)=\sum_{\mathbb{X}_*}k_1 M_3(k,k_1,k_2,k_3)\int_0^t e^{-i(t-s)k^2}[1-\widehat{\eta}(\Delta(t-s))]\overline{\widehat{v_1}(s,k_1)}\widehat{v_2}(s,k_2)\widehat{v_3}(s,k_3)\,\mathrm{d}s
\end{equation}($X$ for ``extra''). Clearly $I\mathcal{C}_*=E_*=E_*^X+E_*^Y$. As with $I$, we will also define the time truncated versions
\begin{equation}\label{timetrunc}\mathcal{E}_*^Y(v_1,v_2,v_3)=\varphi(t)\cdot E_*^Y(\varphi(s)v_1,v_2,v_3),\quad \mathcal{E}_*^X(v_1,v_2,v_3)=\varphi(t)\cdot E_*^X(\varphi(s)v_1,v_2,v_3).
\end{equation}
\begin{prop}\label{newexp} For $*\in\{N,L\}$ we have the expressions:
\begin{equation}\label{formforxy}
\begin{aligned}\mathscr{X}\mathcal{E}_*^Y(v_1,v_2,v_3)(k,\lambda)&=\sum_{\mathbb{X}_*}k_1M_3(k,k_1,k_2,k_3)\int_{\mathbb{R}}K_{\Delta}^{Y}(\lambda,\sigma)\,\mathrm{d}\sigma\int_{\lambda_2+\lambda_3-\lambda_1=\sigma-\Delta}\overline{\widetilde{v_1}(k_1,\lambda_1)}\widetilde{v_2}(k_2,\lambda_2)\widetilde{v_3}(k_3,\lambda_3),\\
\mathscr{X}\mathcal{E}_*^X(v_1,v_2,v_3)(k,\lambda)&=\sum_{\mathbb{X}_*}k_1M_3(k,k_1,k_2,k_3)\int_{\mathbb{R}}K_{\Delta}^{X}(\lambda,\sigma)\,\mathrm{d}\sigma\int_{\lambda_2+\lambda_3-\lambda_1=\sigma-\Delta}\overline{\widetilde{v_1}(k_1,\lambda_1)}\widetilde{v_2}(k_2,\lambda_2)\widetilde{v_3}(k_3,\lambda_3),
\end{aligned}
\end{equation} where recall $\Delta=k^2+k_1^2-k_2^2-k_3^2$, and the functions $K_\Delta^Y$ and $K_\Delta^X$ satisfy the bounds
\begin{equation}\label{bdd1}|K_\Delta^Y(\lambda,\sigma)|\lesssim_B\frac{1}{\langle \lambda-\sigma\rangle^B}\min\bigg(\frac{1}{\langle\Delta\rangle},\frac{1}{\langle \sigma\rangle}\bigg)+\frac{1}{\langle\lambda-\sigma\rangle}\min\bigg(\frac{1}{\langle\Delta\rangle},\frac{1}{\langle \lambda\rangle}\bigg),
\end{equation} 
\begin{equation}\label{bdd2}|K_\Delta^X(\lambda,\sigma)|\lesssim_B\frac{1}{\langle \lambda\rangle^B\langle\sigma\rangle}+\frac{\langle\sigma-\Delta\rangle}{\langle\lambda-\sigma\rangle^B\langle\sigma\rangle}\min\bigg(\frac{1}{\langle\Delta\rangle},\frac{1}{\langle \sigma\rangle}\bigg)+\frac{\langle\lambda-\Delta\rangle}{\langle\lambda-\sigma\rangle}\min\bigg(\frac{1}{\langle\Delta\rangle},\frac{1}{\langle \lambda\rangle}\bigg)^2.
\end{equation}
\end{prop}
\begin{proof} Fix $*\in\{N,L\}$. Let $K_{\Delta}^Y$ be the integral kernel of the linear operator
\[F(s)\mapsto\int_0^t\eta(\Delta(t-s))F(s)\,\mathrm{d}s.\] on the Fourier side, i.e.
\[\mathscr{F}_t\bigg(\int_0^t\eta(\Delta(t-s))F(s)\,\mathrm{d}s\bigg)(\lambda)=\int_{\mathbb{R}}K_{\Delta}^Y(\lambda,\sigma)\widehat{F}(\sigma)\,\mathrm{d}\sigma,\] and $K_\Delta^X=K-K_\Delta^Y$ where $K$ is defined in (\ref{linfor1}). By making Fourier expansion in $x$ twisting by $e^{\pm itk^2}$ on the time-Fourier side at mode $k$, one can see that (\ref{formforxy}) holds with exactly the same kernels $K_\Delta^Y$ and $K_\Delta^X$.

It then suffices to calculate these kernels; by an argument similar to \cite{Deng}, Lemma 3.3, we have
\begin{equation}\label{kdeltay}K_\Delta^Y(\lambda,\sigma)=i\int_{\mathbb{R}}\widehat{\varphi}(\lambda-\mu)\widehat{\varphi}(\mu-\sigma)\frac{1}{\Delta}(\mathcal{H}\widehat{\eta})\bigg(\frac{\mu}{\Delta}\bigg)\,\mathrm{d}\mu-i\int_{\mathbb{R}}\widehat{\varphi}(\lambda-\mu)\frac{1}{\Delta}\widehat{\eta}\bigg(\frac{\mu}{\Delta}\bigg)(\mathcal{H}\widehat{\varphi})(\mu-\sigma)\,\mathrm{d}\mu.
\end{equation} Since $\widehat{\eta}$ and $\widehat{\varphi}$ are Schwartz functions, their Hilbert transforms will decay like $\langle \lambda\rangle^{-1}$, thus
\[\bigg|\frac{1}{\Delta}(\mathcal{H}\widehat{\eta})\bigg(\frac{\mu}{\Delta}\bigg)\bigg|\lesssim\min\bigg(\frac{1}{\langle\Delta\rangle},\frac{1}{\langle \mu\rangle}\bigg),\quad |(\mathcal{H}\widehat{\varphi})(\mu-\sigma)|\lesssim\frac{1}{\langle \mu-\sigma\rangle}.\] Then, by elementary estimates of the integral, the first term on the right hand side of (\ref{kdeltay}) is bounded by the first term on the right hand side of (\ref{bdd1}), and the second term on the right hand side (\ref{kdeltay}) is bounded by the second term on the right hand side of (\ref{bdd1}).

As with $K_{\Delta}^X$, using (\ref{kdeltay}) and (\ref{linfor1}) we can calculate
\begin{multline}\label{kdeltax} K_{\Delta}^X(\lambda,\sigma)=i\int_{\mathbb{R}}\widehat{\varphi}(\lambda-\mu)\widehat{\varphi}(\mu-\sigma)\bigg[\frac{1}{\mu}-\frac{1}{\Delta}(\mathcal{H}\widehat{\eta})\bigg(\frac{\mu}{\Delta}\bigg)\bigg]\,\mathrm{d}\mu\\+i\int_{\mathbb{R}}\widehat{\varphi}(\lambda-\mu)\frac{1}{\Delta}\widehat{\eta}\bigg(\frac{\mu}{\Delta}\bigg)(\mathcal{H}\widehat{\varphi})(\mu-\sigma)\,\mathrm{d}\mu-i\widehat{\varphi}(\lambda)\int_{\mathbb{R}}\frac{\widehat{\varphi}(\mu-\sigma)}{\mu}\,\mathrm{d}\mu.
\end{multline} The third term on the right hand side of (\ref{kdeltax}) is bounded by the first term on the right hand side of (\ref{bdd2}). The first term on the right hand side of (\ref{kdeltax}) can be bounded by the second term on the right hand side of (\ref{bdd2}), once we can prove
\[\bigg|\frac{1}{\mu}-\frac{1}{\Delta}(\mathcal{H}\widehat{\eta})\bigg(\frac{\mu}{\Delta}\bigg)\bigg|\lesssim\frac{\langle\mu-\Delta\rangle}{\langle\mu\rangle}\min\bigg(\frac{1}{\langle\Delta\rangle},\frac{1}{\langle \mu\rangle}\bigg)\] for $|\mu|\geq 1$, but this follows from rescaling and the assumption $\mathcal{H}\widehat{\eta}(1)=1$. Similarly, the second term on the right hand side of (\ref{kdeltax}) can be bounded by the third term on the right hand side of (\ref{bdd2}), due to the estimate
\[\bigg|\frac{1}{\Delta}\widehat{\eta}\bigg(\frac{\mu}{\Delta}\bigg)\bigg|\lesssim\langle \mu-\Delta\rangle\min\bigg(\frac{1}{\langle\Delta\rangle},\frac{1}{\langle \mu\rangle}\bigg)^2\] and the fact that $\widehat{\eta}(1)=0$.
\end{proof}
\begin{rem}\label{further} Note that the first term on the right hand side of (\ref{bdd1}) is bounded by the second term, so we have 
\begin{equation}\label{kybd}|K_\Delta^{Y}|\lesssim \frac{1}{\langle\lambda-\sigma\rangle}\min\bigg(\frac{1}{\langle\Delta\rangle},\frac{1}{\langle \lambda\rangle}\bigg).
\end{equation}Moreover, by (\ref{bdd2}) we can write $K_\Delta^X=K_\Delta^{X,0}+K_\Delta^{X,+}$, where
\begin{equation}\label{k0bd}|K_\Delta^{X,0}|\lesssim\mathbf{1}_{\langle \sigma\rangle\gtrsim\langle\Delta\rangle}\frac{1}{\langle\lambda\rangle^B\langle\Delta\rangle}+\mathbf{1}_{\langle \lambda-\sigma\rangle\gtrsim\langle\sigma-\Delta\rangle}\cdot\min\bigg(\frac{1}{\langle \Delta\rangle},\frac{1}{\langle \lambda\rangle}\bigg)^2,\end{equation}
\begin{multline}\label{k+bd}|K_\Delta^{X,+}|\lesssim_B \mathbf{1}_{\langle\sigma\rangle\ll\langle\Delta\rangle}\frac{1}{\langle \lambda\rangle^B\langle\sigma\rangle}+\frac{\langle\sigma-\Delta\rangle}{\langle\lambda-\sigma\rangle^B\langle\sigma\rangle}\min\bigg(\frac{1}{\langle\Delta\rangle},\frac{1}{\langle \sigma\rangle}\bigg)\\+\mathbf{1}_{\langle \lambda-\sigma\rangle\ll\langle\sigma-\Delta\rangle}\frac{\langle\sigma-\Delta\rangle}{\langle\lambda-\sigma\rangle}\min\bigg(\frac{1}{\langle\Delta\rangle},\frac{1}{\langle \lambda\rangle}\bigg)^2.
\end{multline} We will define the terms $\mathcal{E}_*^{X,0}$ and $\mathcal{E}_*^{X,+}$ accordingly, for $*\in\{N,L\}$.
\end{rem}
\subsection{The submanifold $\mathcal{W}$} We can now define $\mathcal{W}$ as follows.
\begin{df}\label{defset} Recall that $A_2$, $A_3$ and $T$ are fixed. Let $J=[-T,T]$. We define
\begin{multline}\label{defw}\mathcal{W}=\big\{v\in Y_0(I):\|v\|_{Y_0(I)}\leq A_3,\textrm{ and }\textrm{ there exists }w\textrm{ with }\|w\|_{Z_0(I)}\leq A_2,\\\textrm{such that }v=w+E_{N}^{Y}(w,w,v)+E_{L}^Y(w,v,v)\big\}.
\end{multline} This is a submanifold of $Y_0(I)\subset C_t^0H_{p_0}^{\frac{1}{2}}(I)$. Moreover, we will define the submanifold $\mathcal{Z}$ of $C_t^0H_{p_0}^{\frac{1}{2}}(I)$ in the statement of Theorem \ref{main} by $\mathcal{Z}=\mathcal{G}^{-1}\mathcal{W}$.
\end{df} We will need the following proposition, whose proof is postponed to Section \ref{recoverv}.
\begin{prop}\label{cont} For any $w$ which satisfies $\|w\|_{Z_0(I)}\leq A_2$ there is a unique $v$ satisfying $\|v\|_{Y_0(I)}\leq A_3$, such that
\begin{equation}\label{ansatz}v=w+E_{N}^{Y}(w,w,v)+E_{L}^Y(w,v,v).
\end{equation} This mapping $w\mapsto v=v[w]$ is Lipschitz from the $A_2$-ball of $Z_0(I)$ to the $A_3$-ball of $Y_0(I)$. The submanifold $\mathcal{W}$ of $Y_0(I)$ is the image of this mapping.
\end{prop}
\subsection{Reducing to an equation for $w$} The next step is to reduce (\ref{intg2}) to an equation for $w$. We will construct a function $w$ satisfying $\|w\|_{Z_0(I)}\leq A_2$, such that the function $v=v[w]$ defined by Proposition \ref{cont} satisfies (\ref{intg2}). By direct calculation, we see that (\ref{intg2}) reduces to
\begin{equation}\label{dnls5}
\begin{aligned} w=e^{it\partial_x^2}v_0&+I\mathcal{Q}(v,\cdots,v)+I\mathcal{C}_H(v,v,v)+I\mathcal{C}_S(v,v,v)\\
&+I(\mathcal{C}_{N}(v,v,v)-\mathcal{C}_{N}(w,w,v))+I(\mathcal{C}_{L}(v,v,v)-\mathcal{C}_{L}(w,v,v))\\
&+E_{N}^{X}(w,w,v)+E_{L}^{X}(w,v,v).
\end{aligned}
\end{equation} where $v=v[w]$ (we will always assume this below) and satisfies
\begin{equation}\label{ansatz2}v=w+E_{N}^{Y}(w,w,v)+E_{L}^Y(w,v,v).
\end{equation} It is now clear that Theorem \ref{main} will be a consequence of the following
\begin{prop}\label{main2} The mapping that maps $w$ to the right hand side of (\ref{dnls5}) is a contraction mapping from the $A_2$-ball of $Z_0(I)$ to itself.
\end{prop} This proposition will be proved in Section \ref{recoverw}.
\section{Proof of Proposition \ref{cont}}\label{recoverv} 
In  this section  we prove a prior bounds for the para-controlling terms which will crucially enter in the next section
We start by noting that $Z_0(I)\subset Y_0(I)$. In order to prove Proposition \ref{cont}, it suffices to prove the trilinear estimates
\begin{align}\label{trilineary1}\|E_N^Y(v_1,v_2,v_3)\|_{Y_{0}(I)}&\lesssim T^{\theta}\|v_1\|_{Z_0(I)}\|v_2\|_{Z_0(I)}\|v_3\|_{Y_0(I)},\\
\label{trilineary2}\|E_L^Y(v_1,v_2,v_3)\|_{Y_0(I)}&\lesssim T^{\theta}\|v_1\|_{Z_0(I)}\|v_2\|_{Y_0(I)}\|v_3\|_{Y_0(I)}.
\end{align} In fact, these would imply that given $w$ which satisfies $\|w\|_{X_{p,q}^{\frac{1}{2},1}(I)}\leq A_2$, the mapping
\[v\mapsto w+E_N^Y(w,w,v)+E_L^Y(w,v,v)\] is a contraction mapping from the $A_3$-ball of $Y_0(I)$ to itself. It then has a unique fixed point $v=v[w]$, and the Lipschitz property of the mapping $w\mapsto v$ is also easily checked.

In order to prove (\ref{trilineary1}) and (\ref{trilineary2}), we will assume that $w^+$ and $v^+$ are extensions of $w$ and $v$ respectively, such that $\|w^+\|_{Z_0}\leq 2A_2$ and $\|v\|_{Y_0}\leq 2A_3$. Recall that $\varphi_T(t)=\varphi(T^{-1}t)$, clearly $\varphi_T\cdot \mathcal{E}_N^Y(w^+,w^+,v^+)$ and $\varphi_T\cdot \mathcal{E}_L^Y(w^+,v^+,v^+)$ are extensions of $E_N^Y(w,w,v)$ and $E_L^Y(w,v,v)$ respectively. Using also Proposition \ref{stcut}, we can reduce Proposition \ref{cont} to the following
\begin{prop}\label{contnew} We have the following bounds
\begin{align}\label{trilineary3}\|\mathcal{E}_N^Y(v_1,v_2,v_3)\|_{Y_1}&\lesssim \|v_1\|_{Z_0}\|v_2\|_{Z_0}\|v_3\|_{Y_0},\\
\label{trilineary4}\|\mathcal{E}_L^Y(v_1,v_2,v_3)\|_{Y_1}&\lesssim \|v_1\|_{Z_0}\|v_2\|_{Y_0}\|v_3\|_{Y_0}.
\end{align}
\end{prop}
\begin{proof}Let $*\in\{N,L\}$, using the embedding $Z_0\subset Y_0$, we only need to prove the stronger result 
\begin{equation}\label{trilineary5}\|\mathcal{E}_*^Y(v_1,v_2,v_3)\|_{Y_1}\lesssim \|v_1\|_{Z_0}\|v_2\|_{Y_0}\|v_3\|_{Y_0}.\end{equation}Let $\mathcal{E}=\mathcal{E}_*^Y(v_1,v_2,v_3)$, we may assume the norms on the right hand side are all equal to $1$. Recall from (\ref{formforxy}) and (\ref{kybd}) that
\begin{equation}\label{expe}|\widetilde{\mathcal{E}}(k,\lambda)|\lesssim\sum_{\mathbb{X}_*}|k_1|\min\bigg(\frac{1}{\langle \Delta\rangle},\frac{1}{\langle \lambda\rangle}\bigg)\int_{\lambda_2+\lambda_3+\lambda_4-\lambda_1=\lambda-\Delta}\frac{1}{\langle \lambda_4\rangle}\prod_{j=1}^3|\widetilde{v_j}(k_j,\lambda_j)|,
\end{equation} where $\Delta=2(k-k_2)(k-k_3)$ as before. we may restrict to the dyadic region $\langle k_2\rangle\sim N_2$ and $\langle k_3\rangle\sim N_3$ (so $N_2\gtrsim N_3$), where $N_2$ and $N_3$ are powers of two.

Recall that $\frac{1}{r_2}=(\frac{1}{2})+3\delta$ (so $r_2<r_0$). Notice that 
\[\|\langle k_1\rangle^{\frac{1}{2}}\widetilde{v_1}\|_{L_{\lambda}^1\ell_k^{p_0}}\lesssim\|\langle k_1\rangle^{\frac{1}{2}}\langle\lambda_1\rangle^{b_0}\widetilde{v_1}\|_{L_{\lambda}^{q_0}\ell_k^{p_0}}\lesssim\|\langle k_1\rangle^{\frac{1}{2}}\langle\lambda_1\rangle^{b_0}\widetilde{v_1}\|_{\ell_k^{p_0}L_{\lambda}^{q_0}}\lesssim 1\]by H\"{o}lder and Minkowski, and similarly
\begin{multline*}\|\langle k_j\rangle^{(1-\sqrt{\delta})/p_0}\widetilde{v_j}\|_{L_\lambda^1\ell_k^{r_2}}\lesssim \|\langle k_j\rangle^{(1-\sqrt{\delta})/p_0}\langle\lambda_j\rangle^{\frac{1}{2}}\widetilde{v_j}\|_{L_\lambda^{r_0}\ell_k^{r_2}}\\\lesssim\|\langle k_j\rangle^{(1-\sqrt{\delta})/p_0}\langle\lambda_j\rangle^{\frac{1}{2}}\widetilde{v_j}\|_{\ell_k^{r_2}L_\lambda^{r_0}}\lesssim\|\langle k_j\rangle^{\frac{1}{2}}\langle \lambda_j\rangle^{\frac{1}{2}}\widetilde{v_j}\|_{\ell_k^{p_0}L_\lambda^{r_0}}\lesssim 1\end{multline*} for $j\in\{2,3\}$, we may then fix $(\lambda_1,\lambda_2,\lambda_3)$ which we eventually intergate over, and denote \[|\widetilde{v_1}(k_1,\lambda_1)|=\langle k_1\rangle^{-\frac{1}{2}}f_1(k_1),\quad|\widetilde{v_j}(k_j,\lambda_j)|=N_j^{-(1-\sqrt{\delta})/p_0}f_j(k_j),\,\,2\leq j\leq 3,\] where (after a further normalization)
\begin{equation}\label{bddf}\|f_1\|_{\ell_k^{p_0}}\lesssim 1,\quad \|f_j\|_{\ell_k^{r_2}}\lesssim 1\,\,(k=2,3),\end{equation} and it will suffice to prove that for any fixed $\mu(=\lambda_2+\lambda_3-\lambda_1)\in\mathbb{R}$,
\begin{equation}\label{red1}\bigg\|\langle k\rangle^{\frac{1}{2}}\langle \lambda\rangle^{\frac{1}{2}}\sum_{\mathbb{X}_*}\langle k_1\rangle^{\frac{1}{2}}\min\bigg(\frac{1}{\langle \Delta\rangle},\frac{1}{\langle \lambda\rangle}\bigg)\frac{1}{\langle \lambda-\Delta-\mu\rangle}\prod_{j=1}^3f_j(k_j)\bigg\|_{\ell_k^{p_0}L_\lambda^{r_1}}\lesssim (N_2N_3)^{(1-\sqrt{\delta})/p_0-\theta}.
\end{equation} In the above summation over $(k_1,k_2,k_3)\in \mathbb{X}_*$, we may first fix $\Delta$ and sum over $(k_1,k_2,k_3)\in \mathbb{X}_*$ that corresponds to this fixed $\Delta$. 

We first assume $*=L$, which is the slightly harder case. Note that $\langle k\rangle\sim\langle k_1\rangle$, by Lemma \ref{newdiv}, we can bound the left hand side of (\ref{red1}) by
\begin{equation}\label{red1exp}\bigg\|\langle k\rangle\langle \lambda\rangle^{\frac{1}{2}}\sum_{\Delta}F(k,\Delta)\min\bigg(\frac{1}{\langle \Delta\rangle},\frac{1}{\langle \lambda\rangle}\bigg)\frac{1}{\langle \lambda-\Delta-\mu\rangle}\bigg\|_{\ell_k^{p_0}L_\lambda^{r_1}},
\end{equation} where
\[F(k,\Delta)=\sum_{\substack{(k_1,k_2,k_3)\in \mathbb{X}_L\\k^2+k_1^2-k_2^2-k_3^2=\Delta}}\prod_{j=1}^3f_j(k_j)\lesssim N_2^\theta\bigg(\sum_{\substack{(k_1,k_2,k_3)\in \mathbb{X}_L\\k^2+k_1^2-k_2^2-k_3^2=\Delta}}\prod_{j=1}^3f_j(k_j)^{r_2}\bigg)^{\frac{1}{r_2}}.\] Using the facts that \[\min\bigg(\frac{1}{\langle \Delta\rangle},\frac{1}{\langle \lambda\rangle}\bigg)\lesssim\langle \Delta\rangle^{-\frac{1}{2}}\langle\lambda\rangle^{-\frac{1}{2}},\quad \langle \Delta\rangle\sim \langle k\rangle^2\]and by Schur's estimate, we can bound
\[\bigg\|\langle k\rangle\langle \lambda\rangle^{\frac{1}{2}}\sum_{\Delta}F(k,\Delta)\min\bigg(\frac{1}{\langle \Delta\rangle},\frac{1}{\langle \lambda\rangle}\bigg)\frac{1}{\langle \lambda-\Delta-\mu\rangle}\bigg\|_{L_\lambda^{r_1}}\lesssim\|F(k,\Delta)\|_{\ell_{\Delta}^{r_2}}\] for each fixed $k$. By the definition of $F(k,\Delta)$, it then suffices to prove that
\begin{equation}\label{newadd}N_2^\theta\bigg\|\bigg(\sum_{\mathbb{X}_L}\prod_{j=1}^3f_j(k_j)^{r_2}\bigg)^{\frac{1}{r_2}}\bigg\|_{\ell_k^{p_0}}\lesssim (N_2N_3)^{(1-\sqrt{\delta})/p_0}N_2^{-\theta}.\end{equation} Let $f_j(k_j)^{r_2}=g_j(k_j)$ and $\beta=(p_0/r_2)$, it suffices to prove (for a possibly different $\theta$) that
\[\bigg\|\sum_{\substack{|k_2|\sim N_2\\|k_3|\sim N_3}}g_1(k_2+k_3-k)g_2(k_2)g_3(k_3)\bigg\|_{\ell_k^{\beta}}\lesssim (N_2N_3)^{r_2(1-\sqrt{\delta})/p_0} N_2^{-\theta}.\] As $\|g_1\|_{\ell_k^{\beta}}=\|f_1\|_{\ell_k^{p_0}}^{r_2}\lesssim1$, by Minkowski we can bound the above by
\begin{equation}\label{neg}\|g_2\|_{\ell_k^1}\|g_3\|_{\ell_k^1}=\|f_2\|_{\ell_k^{r_2}}^{r_2}\|f_3\|_{\ell_k^{r_2}}^{r_2}\lesssim 1\end{equation} using (\ref{bddf}). This finishes the case $*=L$.

When $*=N$, we will further assume $\langle k\rangle\sim N_0$ and $\langle k_1\rangle\sim N_1$, then all the proof will be the same as above, using the fact that
\[\langle k\rangle^{\frac{1}{2}}\langle k_1\rangle^{\frac{1}{2}}\sim N_0N_1\sim\langle \Delta\rangle^{\frac{1}{2}}.\] The sum over $N_0$ and $N_1$ is then taken care of using the positive power of $N_2$ on the right hand side of (\ref{newadd}), and the fact that $N_2\gtrsim \max(N_0,N_1)$ when $(k_1,k_2,k_3)\in\mathbb{X}_N$.
\end{proof}
\section{Proof of Proposition \ref{main2}}\label{recoverw} 

This section constitutes the heart of the paper. Here we find $w$ and prove all the underlying multilinear estimates involved in its construction. In the course of the proof we show in particular that all relevant nonlinearities are well defined as space-time distributions whence the integral equation \eqref{intg} for $u$ will be equivalent to the integral equation formulation of \eqref{gdnls} from Section \ref{setup}.

Given $w$ satisfying $\|w\|_{Z_0(I)}\leq A_2$, let $w^+$ be an extension of $w$ such that $\|w^+\|_{Z_0}\leq 2A_2$. By the proof of Proposition \ref{cont} in Section \ref{recoverv}, we know that there is a unique $v^+=v^+[w^+]$ such that $\|v^+\|_{Y_0}\leq A_3$, and 
\begin{equation}\label{defvplus}v^+=w^++\varphi_T\cdot\mathcal{E}_N^Y(w^+,w^+,v^+)+\varphi_T\cdot\mathcal{E}_L^Y(w^+,v^+,v^+).
\end{equation} Moreover this $v^+$ is an extension of $v=v[w]$. Therefore, recall that $\psi:=\varphi(t)e^{it\partial_x^2}v_0$, the function
\begin{equation}\label{dnlsext}
\begin{aligned} z:=\psi&+\varphi_T\cdot\mathcal{IQ}(v^+,\cdots,v^+)+\varphi_T\cdot\mathcal{IC}_H(v^+,v^+,v^+)+\varphi_T\cdot\mathcal{IC}_S(v^+,v^+,v^+)\\
&+\varphi_T\cdot\mathcal{I}(\mathcal{C}_{N}(v^+,v^+,v^+)-\mathcal{C}_{N}(w^+,w^+,v^+))+\varphi_T\cdot\mathcal{I}(\mathcal{C}_{L}(v^+,v^+,v^+)-\mathcal{C}_{L}(w^+,v^+,v^+))\\
&+\varphi_T\cdot\mathcal{E}_{N}^{X}(w^+,w^+,v^+)+\varphi_T\cdot\mathcal{E}_{L}^{X}(w^+,v^+,v^+)
\end{aligned}
\end{equation} will be an extension of the right hand side of (\ref{dnls5}).
\subsection{Splitting the formula of $z$}\label{splitz}
Now that $w^+$, $v^+$ and $z$ are defined for all time, we can further manipulate the expression of $z$, as this manipulation sometimes requires inserting time-frequency cutoffs. We will analyze each term in (\ref{dnlsext}) separately. The initial data term $\psi$ is trivial. For the other terms, we will remove the $\varphi_T$ factor in front, and bound the corresponding terms in the stronger space $Z_1$; Proposition \ref{stcut} then allows us to gain a factor $T^\theta$ which provides the required smallness.

(1) \underline{The term $\mathcal{IQ}(v^+,\cdots,v^+)$.} This is a single term, we will name it
\begin{equation}\label{defz1}z_{11}=\mathcal{IQ}(v^+,\cdots,v^+).\end{equation}

(2) \underline{The term $\varphi_T\cdot\mathcal{I}(\mathcal{C}_H+\mathcal{C}_S)(v^+,v^+,v^+)$}. Here decomposing $v^+$ by (\ref{defvplus}), we can obtain the following terms
\begin{equation}\label{defz2}
\begin{aligned}z_{21}&=\mathcal{I}(\mathcal{C}_H+\mathcal{C}_S)(w^+,w^+,w^+),\\ z_{22}&=\mathcal{I}(\mathcal{C}_H+\mathcal{C}_S)(\varphi_T\cdot\mathcal{E}_{N}^{Y}(w^+,w^+,v^+),v^+,v^+),\\
\quad z_{23}&=\mathcal{I}(\mathcal{C}_H+\mathcal{C}_S)(\varphi_T\cdot\mathcal{E}_{L}^{Y}(w^+,v^+,v^+),v^+,v^+),\\
z_{24}&=\mathcal{I}(\mathcal{C}_H+\mathcal{C}_S)(w^+,\varphi_T\cdot\mathcal{E}_{N}^{Y}(w^+,w^+,v^+),v^+),\\
z_{25}&=\mathcal{I}(\mathcal{C}_H+\mathcal{C}_S)(w^+,\varphi_T\cdot\mathcal{E}_{L}^{Y}(w^+,v^+,v^+),v^+),\\
z_{26}&=\mathcal{I}(\mathcal{C}_H+\mathcal{C}_S)(w^+,w^+,\varphi_T\cdot\mathcal{E}_{N}^{Y}(w^+,w^+,v^+)),\\
z_{27}&=\mathcal{I}(\mathcal{C}_H+\mathcal{C}_S)(w^+,w^+,\varphi_T\cdot\mathcal{E}_{L}^{Y}(w^+,v^+,v^+)).
\end{aligned}
\end{equation} Here $z_{21}$ is a cubic expression, and the others are quintic expressions.

(3) \underline{The term $\mathcal{I}(\mathcal{C}_{N}(v^+,v^+,v^+)-\mathcal{C}_{N}(w^+,w^+,v^+))$}. Similar to (2), we can obtain the terms
\begin{equation}\label{defz3}
\begin{aligned}z_{31}&=\mathcal{I}(\mathcal{C}_{N}(\varphi_T\cdot\mathcal{E}_{N}^{Y}(w^+,w^+,v^+),v^+,v^+),\\
z_{32}&=\mathcal{I}(\mathcal{C}_{N}(\varphi_T\cdot\mathcal{E}_{L}^{Y}(w^+,v^+,v^+),v^+,v^+),
\\z_{33}&=\mathcal{I}(\mathcal{C}_{N}(w^+,\varphi_T\cdot\mathcal{E}_{N}^{Y}(w^+,w^+,v^+),v^+),\\
z_{34}&=\mathcal{I}(\mathcal{C}_{N}(w^+,\varphi_T\cdot\mathcal{E}_{L}^{Y}(w^+,v^+,v^+),v^+).
\end{aligned}
\end{equation} They are all quintic expressions.

(4) \underline{The term $\mathcal{I}(\mathcal{C}_{L}(v^+,v^+,v^+)-\mathcal{C}_{L}(w^+,v^+,v^+))$}. In the same way we get two terms
\begin{equation}\label{defz4}
\begin{aligned}z_{41}&=\varphi_T\cdot\mathcal{I}(\mathcal{C}_{L}(\mathcal{E}_{N}^{Y}(w^+,w^+,v^+),v^+,v^+),\\
z_{42}&=\mathcal{I}(\mathcal{C}_{L}(\varphi_T\cdot\mathcal{E}_{L}^{Y}(w^+,v^+,v^+),v^+,v^+).
\end{aligned}
\end{equation} They are both quintic expressions.

(5) \underline{The term $\mathcal{E}_{N}^{X}(w^+,w^+,v^+)+\mathcal{E}_{L}^{X}(w^+,v^+,v^+)$.} This term requires a little more care. Let $*\in\{N,L\}$, recall that from Proposition \ref{newexp} and Remark \ref{further}, we have
\begin{multline}
\label{expnew}
\mathscr{X}\mathcal{E}_*^{X,+}(v_1,v_2,v_3)(k,\lambda)=\sum_{\mathbb{X}_*}k_1M_3(k,k_1,k_2,k_3)\int_{\mathbb{R}}K_{\Delta}^{X,+}(\lambda,\sigma)\,\mathrm{d}\sigma\\\int_{\lambda_2+\lambda_3-\lambda_1=\sigma-\Delta}\overline{\widetilde{v_1}(k_1,\lambda_1)}\widetilde{v_2}(k_2,\lambda_2)\widetilde{v_3}(k_3,\lambda_3).
\end{multline} We may further decompose this expression into $\mathcal{E}_*^{X,+}=\mathcal{E}_*^{X,1}+\mathcal{E}_*^{X,2}+\mathcal{E}_*^{X,3}$, where in $\mathcal{E}_*^{X,j}$ we make the restriction\[|\lambda_j|=\max_{1\leq \ell\leq 3}|\lambda_{\ell}|,\quad |\lambda_j|\gtrsim|\sigma-\Delta|.\] Now if $*=N$ and $j=3$, or $*=L$ and $j\in\{2,3\}$, we will make decompose the $v^+$ corresponding to frequency $\lambda_j$ using (\ref{defvplus}). This gives the following terms
\begin{equation}\label{defz5}
\begin{aligned}z_{51}&=(\mathcal{E}_{N}^{X,0}+\mathcal{E}_{N}^{X,1}+\mathcal{E}_{N}^{X,2})(w^+,w^+,v^+),\\
z_{52}&=\mathcal{E}_{N}^{X,3}(w^+,w^+,w^+),\\
z_{53}&=\mathcal{E}_{N}^{X,3}(w^+,w^+,\varphi_T\cdot\mathcal{E}_{N}^{Y}(w^+,w^+,v^+)),\\
z_{54}&=\mathcal{E}_{N}^{X,3}(w^+,w^+,\varphi_T\cdot\mathcal{E}_{L}^{Y}(w^+,v^+,v^+)),\\
z_{55}&=(\mathcal{E}_{L}^{X,0}+\mathcal{E}_{L}^{X,1})(w^+,v^+,v^+),\\
z_{56}&=\mathcal{E}_{L}^{X,2}(w^+,w^+,v^+),\\
z_{57}&=\mathcal{E}_{L}^{X,2}(w^+,\varphi_T\cdot\mathcal{E}_{N}^{Y}(w^+,w^+,v^+),v^+),\\
z_{58}&=\mathcal{E}_{L}^{X,2}(w^+,\varphi_T\cdot\mathcal{E}_{L}^{Y}(w^+,v^+,v^+),v^+),\\
z_{59}&=\mathcal{E}_L^{X,3}(w^+,v^+,w^+),\\
z_{5A}&=\mathcal{E}_L^{X,3}(w^+,v^+,\varphi_T\cdot\mathcal{E}_{N}^{Y}(w^+,w^+,v^+)),\\
z_{5B}&=\mathcal{E}_L^{X,3}(w^+,v^+,\varphi_T\cdot\mathcal{E}_{L}^{Y}(w^+,v^+,v^+)).
\end{aligned}
\end{equation} Some of these are cubic expressions, and some of them are quintic.

(6) \underline{An operation on quintic terms}. Each of the above $z_{j\ell}$'s is a multilinear expression, either cubic or quintic; we will always list its input functions from left to right. Consider now a general quintic term. Let $k$ and $k_j$, where $1\leq j\leq 5$, are the (space) frequencies of the output and input functions, then it will involve a summation
\[\sum_{\pm k_1\cdots\pm k_5=k}(\rm{expression}).\] As with Lemma \ref{newdiv}, we say a \emph{pairing} $(i,j)$ happens, if $k_i=k_j$ and the signs of $k_i$ and $k_j$ in the expression $\pm k_1\cdots\pm k_5$ are the opposite.

For each tuple $(k_j)$, we will choose an index $i\in\{1,\cdots,5\}$ as follows: if there is no pairing, then let $i\in\{1,\cdots,5\}$ be such that $|k_i|$ is the maximum; if there is a pairing, say $(1,2)$, and there is no pairing in $\{3,4,5\}$, then let $i\in\{3,4,5\}$ such that $|k_i|$ is the maximum; if there is a pairing in $\{3,4,5\}$, say $(3,4)$, then let $i=5$. It is clear that we always have $|k_i|\gtrsim|k|$.

This procedure then decomposes this quintic term into five parts; once an $i$ is fixed, and if the input function corresponding to this $i$ in this quintic term happens to be $v^+$ (instead of $w^+$), we will decompose this $v^+$ using (\ref{defvplus}), so that this quintic term is decomposed into a quintic and two septic terms.

(7) \underline{Summary}. Now we have decomposed $z$ into a superposition of multilinear expressions $z_{j\ell}$ (including those coming from step (6) above), either cubic or quintic or septic, with input functions being either $w^+$ or $v^+$. Moreover, if we consider two different $w$ and $w'$, then we may choose extensions $w^+$ and $(w')^+$ such that \[\|(w')^+-w^+\|_{X_{p,q}^{\frac{1}{2},1}}\leq 2\|w'-w\|_{X_{p,q}^{\frac{1}{2},1}(I)}.\] Let $v^+$ and $(v')^+$ be defined from $w^+$ and $(w')^+$ by (\ref{defvplus}), then we also have
\[\|(v')^+-v^+\|_{Y_p^{\frac{1}{2}}}\lesssim\|(w')^+-w^+\|_{X_{p,q}^{\frac{1}{2},1}}\lesssim \|w'-w\|_{X_{p,q}^{\frac{1}{2},1}(I)}.\] Then $z$ and $z'$, which are defined by (\ref{dnlsext}) using $w$ and $w'$, satisfy that $z-z'$ is an extension of the difference of the right hand sides of (\ref{dnls5}) corresponding to $w$ and $w'$. Therefore, in order to prove Proposition \ref{main2}, it will suffice to prove the following
\begin{prop}\label{main3} All these terms $z_{j\ell}$, including those coming from step (6) above, satisfy the multilinear estimates
\[\|z_{j\ell}(v_1,\cdots,v_r)\|_{Z_1}\lesssim\|v_1\|\cdots\|v_r\|,\] where $r\in\{3,5,7\}$, and for each $i$, $v_i$ is measured in the $Z_0$ norm if the corresponding input function in $z_{j\ell}$ is $w^+$, and in the $Y_0$ norm if the input is $v^+$. For example the estimate for $z_{42}$ will be
\[\|z_{42}(v_1,\cdots,v_5)\|_{Z_1}\lesssim \|v_1\|_{Z_0}\prod_{j=2}^5\|v_j\|_{Y_0}.\]
\end{prop}
\begin{rem}\label{extrarem} We make a further remark about the operation in step (6) above. For some quintic terms $z_{j\ell}$ this operation is necessary; for others it is not. However, even in the latter case, performing this operation will not affect the proof: if $z_{j\ell}$ itself satisfies a multilinear estimate where this input function $v^+$ is measured in the $Y_0$ norm, then by Propositions \ref{stcut} and \ref{contnew}, after decomposing this $v^+$ using (\ref{defvplus}), the resulting quintic and septic terms will also satisfy the right multilinear estimate. For example, we will see below that
\[\|z_{26}(v_1,\cdots,v_5)\|_{Z_1}\lesssim\prod_{j=1}^4\|v_j\|_{Z_0}\cdot\|v_5\|_{Y_0}.\] Then, even after performing this operation (with the chosen index $i=5$) we still have
\[
\begin{split}\|z_{26}(v_1,\cdots,v_5)\|_{Z_1}&\lesssim\prod_{j=1}^4\|v_j\|_{Z_0}\cdot\|v_5\|_{Y_0},\\
\|z_{26}(v_1,\cdots,v_4,\varphi_T\cdot\mathcal{E}_N^Y(v_5,v_6,v_7))\|_{Z_1}&\lesssim\prod_{j=1}^6\|v_j\|_{Z_0}\cdot\|v_7\|_{Y_0},\\
\|z_{26}(v_1,\cdots,v_4,\varphi_T\cdot\mathcal{E}_N^Y(v_5,v_6,v_7))\|_{Z_1}&\lesssim\prod_{j=1}^5\|v_j\|_{Z_0}\cdot\|v_6\|_{Y_0}\|v_7\|_{Y_0}.
\end{split}\]
\end{rem} The following subsections are devoted to the proof of Proposition \ref{main3}.
\subsection{Cubic terms} In this subsection we treat the cubic terms, which are $z_{21}$ and the cubic $z_{5*}$ terms. First we deal with $z_{21}$ term in the following Proposition \ref{z21}.
\begin{prop}\label{z21}
 $z_{21}$ is defined in (\ref{defz2}). We have the following bound
\begin{equation}\label{z21:bound}
\|z_{21}(v_1, v_2, v_3)\|_{Z_1}\lesssim \prod_{j=1}^3 \|v_j\|_{Z_0}.
\end{equation}
\end{prop}
\begin{proof}
Let $*\in \{H, S\}$, we need to show the following bound
\begin{equation}
\|\mathcal{I}\mathcal{C}_*(v_1, v_2, v_3)\|_{Z_1}\lesssim  \prod_{j=1}^3 \|v_j\|_{Z_0}.
\end{equation}
We may assume the norms on the right hand side are all equal to 1.
Recall from (\ref{esti}) and (\ref{linfor2}) that for any $B>0$,
\begin{equation}
|\widetilde{\mathcal{I}\mathcal{C}}_*(v_1, v_2, v_3)(k,\lambda)|\lesssim\sum_{\mathbb{X}_*} |k_1| \int_{\lambda_2+\lambda_3-\lambda_1=\sigma-\Delta}\left( \frac{1}{\left\langle \lambda\right\rangle^B}+\frac{1}{\left\langle \lambda-\sigma\right\rangle^B} \right)\frac{1}{\langle\sigma\rangle}\prod_{j=1}^3 |\widetilde{v_j}(k_j, \lambda_j)|,
\end{equation}
where $\Delta = 2(k-k_2)(k-k_3)$ as before.

It will suffice to prove that
\begin{equation}\label{z21:bound1}
\left\|
\langle k \rangle^{\frac{1}{2}}  \langle\lambda\rangle^{b_1} \sum_{\mathbb{X}_*}  |k_1| \int_{\lambda_2+\lambda_3-\lambda_1=\sigma-\Delta} \frac{1}{\left\langle \lambda\right\rangle\langle\lambda-\sigma\rangle} \prod_{j=1}^3 |\widetilde{v_j}(k_j, \lambda_j)|
\right\|_{\ell^{p_0}_k L_\lambda^{q_0}}\lesssim 1
\end{equation} 
by the definition of $Y_1$ norm (\ref{defspaces}) and the following inequality
\begin{equation}
\left( \frac{1}{\left\langle \lambda\right\rangle^B}+\frac{1}{\left\langle \lambda-\sigma\right\rangle^B} \right)\frac{1}{\langle\sigma\rangle}\lesssim \frac{1}{\langle\lambda\rangle\langle\lambda-\sigma\rangle}
\end{equation}
for $B$ large enough.

Recall that $b_0  = 1-2\delta$ and $1/q_0=4\delta$ and hence similar to the proof of Proposition \ref{contnew} we have  for $j \in \{1, 2, 3\}$
\[\|\langle k_j\rangle^{\frac{1}{2}}\langle\lambda_j\rangle^{\delta}\widetilde{v_j}\|_{L_{\lambda}^1\ell_k^{p_0}}\lesssim\|\langle k_j\rangle^{\frac{1}{2}}\langle\lambda_j\rangle^{b_0}\widetilde{v_j}\|_{L_{\lambda}^{q_0}\ell_k^{p_0}}\lesssim\|\langle k_j\rangle^{\frac{1}{2}}\langle\lambda_j\rangle^{b_0}\widetilde{v_j}\|_{\ell_k^{p_0}L_{\lambda}^{q_0}}\lesssim 1\]by H\"{o}lder and Minkowski, we may then fix $(\lambda_1, \lambda_2, \lambda_3)$ which we eventually integrate over, and denote
\[
|\widetilde{v_j}(k_j, \lambda_j)| = \langle k_j\rangle^{-\frac{1}{2}}\langle\lambda_j\rangle^{-\delta} f_j(k_j) \,\, (1\leq j\leq 3)
\]
 and it will suffice to prove that  for any fixed  $\mu (=\lambda_2+\lambda_3-\lambda_1)\in \mathbb{R}$,
\begin{equation}\label{z21:bound2}
\left\|
\frac{\langle k \rangle^{\frac{1}{2}}}{\left\langle \lambda\right\rangle^{\delta}} \sum_{\mathbb{X}_*} \frac{ | k_1|}{\langle k_1\rangle^{\frac{1}{2}}\langle k_2\rangle^{\frac{1}{2}} \langle k_3\rangle^{\frac{1}{2}}}  \frac{1}{\langle\lambda-\Delta-\mu\rangle \langle\lambda_1\rangle^{\delta}\langle\lambda_2\rangle^{\delta}\langle\lambda_3\rangle^{\delta}}\prod_{j=1}^3 |f_j| 
\right\|_{\ell^{p_0}_k L_\lambda^{q_0}}\lesssim 1,
\end{equation} 
and then applying the inequality (\ref{semilinear}), it will suffice to prove 
\begin{equation}\label{z21:bound3}
\left\|
\frac{1}{\left\langle \lambda\right\rangle^{\delta}} \sum_{\mathbb{X}_*} \frac{1}{\langle\lambda-\Delta-\mu\rangle \langle\lambda_1\rangle^{\delta}\langle\lambda_2\rangle^{\delta}\langle\lambda_3\rangle^{\delta}}\prod_{j=1}^3 |f_j| 
\right\|_{\ell^{p_0}_k L_\lambda^{q_0}}\lesssim 1,
\end{equation} 
In the above summation over $(k_1, k_2, k_3)\in \mathbb{X}_*$, we again first fix $\Delta$ and sum over $(k_1, k_2, k_3)\in \mathbb{X}_*$ corresponds to this fixed $\Delta$. Using the fact that
\begin{equation}
\left(\langle\lambda_1\rangle\langle\lambda_3\rangle\langle\lambda_3\rangle\langle\lambda\rangle
\langle\lambda-\Delta-\mu\rangle\right)^{\delta} \gtrsim \langle\Delta\rangle^{\delta}
\end{equation}
and by the standard divisor bound\footnote{The divisor bound applies when $\Delta\neq 0$; however when $\Delta=0$ we must have $k=k_1=k_2=k_3$ by the definition of $\mathbb{V}_3$, so the bound is still true.}, we can bound the left side of (\ref{z21:bound3}) by
\begin{equation}\label{z21:bound4}
\left\|
\sum_{\Delta}  \frac{1}{\langle\lambda-\Delta-\mu\rangle^{1-\delta} \langle\Delta\rangle^\delta} F (k, \Delta)
\right\|_{\ell^{p_0}_k L_\lambda^{q_0}}
\end{equation}
where
\begin{equation}\label{z21:bound5}
F(k, \Delta) =\sum_{\substack{(k_1,k_2,k_3)\in\mathbb{X}_*\\2(k-k_2)(k-k_3)=\Delta}} \prod_{j=1}^3 |f_j|\lesssim \langle \Delta \rangle^{\theta} \left(\sum_{\substack{(k_1,k_2,k_3)\in\mathbb{X}_*\\2(k-k_2)(k-k_3)=\Delta}} \prod_{j=1}^3 |f_j|^{p_0}\right)^{\frac{1}{p_0}}.
\end{equation}
By our choice we have $\delta<\frac{1}{5p_0}$ and $\theta<\delta$, so by Schur's estimate, we can bound
\begin{equation}\label{z21:bound6}
\left\|
\sum_{\Delta}  \frac{1}{\langle\lambda-\Delta-\mu\rangle^{1-\delta} \langle\Delta\rangle^\delta} F (k, \Delta)
\right\|_{L_\lambda^{q_0}}\lesssim \left\|\frac{F(k, \Delta)}{\langle\Delta\rangle^{\delta}}\right\|_{\ell^{p_0}_\Delta}.
\end{equation}
Then we may sum over $k$ and we obtain that 
\begin{equation}\label{z21:bound7}
\left\|
\sum_{\Delta}  \frac{1}{\langle\lambda-\Delta-\mu\rangle^{1-\delta} \langle\Delta\rangle^\delta} F (k, \Delta)
\right\|_{\ell^{p_0}_k L_\lambda^{q_0}}\lesssim \prod_{j=1}^3 \|f_j\|_{\ell^{p_0}_k}
\end{equation}
by (\ref{z21:bound5}) and (\ref{z21:bound6}). Finally we integrate (\ref{z21:bound7}) over $(\lambda_1, \lambda_2, \lambda_3)$ and it finishes this proof.
\end{proof}

Next let's consider the cubic $z_{5*}$ terms (i.e. $z_{51}$, $z_{52}$, $z_{55}$, $z_{56}$ and $z_{59}$). The following Proposition \ref{z5} gives the suitable bounds for $z_{51}$, $z_{52}$, $z_{55}$, $z_{56}$ and $z_{59}$ in Proposition \ref{main3}.
\begin{prop}\label{z5}
For $*\in \{N, L\}$, $j\in \{0, 1, 2, 3\}$ and $\mathcal{E}_{*}^{X,j}$ defined in (\ref{formforxy}) and the description above (\ref{defz5}), we have the following bounds.

(1) If $j = 0$, we obtain that 
\begin{equation}\label{z5:X0}
\left\|\mathcal{E}_{*}^{X,0}(v_1, v_2, v_3)\right\|_{Z_1}\lesssim \|v_1\|_{Z_0}\|v_2\|_{Y_0} \|v_3\|_{Y_0}.
\end{equation}

(2) If $j = 1$, we obtain that 
\begin{equation}\label{z5:X1}
\left\|\mathcal{E}_{*}^{X,1}(v_1, v_2, v_3)\right\|_{Z_1}\lesssim \|v_1\|_{Z_0}\|v_2\|_{Y_0} \|v_3\|_{Y_0}.
\end{equation}

(3) If $j \in \{2, 3\}$ and $i\in\{2, 3\}-\{j\}$, we obtain that
\begin{equation}\label{z5:X23}
\left\|\mathcal{E}_{*}^{X,j}(v_1, v_2, v_3)\right\|_{Z_1}\lesssim \|v_1\|_{Z_0}\|v_j\|_{Z_0} \|v_i\|_{Y_0}
\end{equation}
 \end{prop}
\begin{proof}
Recall from (\ref{formforxy}) that
\begin{equation}
|\widetilde{\mathcal{E}_{*}^{X,0}}(v_1, v_2, v_3)(k,\lambda)|\lesssim
\sum_{\mathbb{X}_*} |k_1| \int_{\lambda_2+\lambda_3-\lambda_1=\sigma-\Delta} |K_\Delta^{X, 0}(\lambda,\sigma)| \prod_{j=1}^3 |\widetilde{v_j}(k_j, \lambda_j)|,
\end{equation}
and for $j\in\{1, 2, 3\}$
\begin{equation}
|\widetilde{\mathcal{E}_{*}^{X,j}}(v_1, v_2, v_3)(k,\lambda)|\lesssim
\sum_{\mathbb{X}_*} |k_1| \int_{\substack{\lambda_2+\lambda_3-\lambda_1=\sigma-\Delta\\|\lambda_j|=\max_{1\leq \ell\leq 3}|\lambda_\ell|}} |K_\Delta^{X, +}(\lambda,\sigma)| \prod_{j=1}^3 |\widetilde{v_j}(k_j, \lambda_j)|,
\end{equation}
where $\Delta = 2(k-k_2)(k-k_3)$ as before. 

(1) Let's consider the case when $j = 0$ and $*\in \{N, L\}$, and then left side of the bound (\ref{z5:X0}) can be bounded by
\begin{equation}\label{z5:norm0}
\left\|
\langle k \rangle^{\frac{1}{2}} \langle \lambda\rangle^{b_1} \sum_{\mathbb{X}_*}  |k_1| \int_{\lambda_2+\lambda_3-\lambda_1=\sigma-\Delta} |K_\Delta^{X, 0}(\lambda,\sigma)| \prod_{j=1}^3 |\widetilde{v_j}(k_j, \lambda_j)|
\right\|_{\ell^{p_0}_k L_\lambda^{q_0}}.
\end{equation}
Recall (\ref{k0bd}), it will suffice to prove that
\begin{equation}\label{z5:norm01}
\left\|
\langle k \rangle^{\frac{1}{2}} \langle \lambda\rangle^{b_1} \sum_{\mathbb{X}_*}  |k_1| \int_{\lambda_1, \lambda_2, \lambda_3} \frac{1}{\langle\lambda\rangle^{1+4\delta}\langle\Delta\rangle^{1-4\delta}}
 \prod_{j=1}^3 |\widetilde{v_j}(k_j, \lambda_j)|
\right\|_{\ell^{p_0}_k L_\lambda^{q_0}}\lesssim \|v_1\|_{Z_0}\|v_2\|_{Y_0} \|v_3\|_{Y_0}
\end{equation} 
and then by Minkowski's inequality and integrating over $\lambda$ the left side of (\ref{z5:norm01}) can bounded by
\begin{equation}
\left\|
\langle k \rangle^{\frac{1}{2}} \sum_{\mathbb{X}_*}  |k_1| \int_{\lambda_1, \lambda_2, \lambda_3} \frac{1}{\langle\Delta\rangle^{1-4\delta}}
 \prod_{j=1}^3 |\widetilde{v_j}(k_j, \lambda_j)|
\right\|_{\ell^{p_0}_k}.
\end{equation}
We may then fix $(\lambda_1,\lambda_2,\lambda_3)$ which we eventually integrate over.
 In the above summation over $(k_1, k_2, k_3)\in \mathbb{X}_*$, we may first fix $\Delta$ and sum over $(k_1, k_2, k_3)\in \mathbb{X}_*$ that corresponds to this fixed $\Delta$. Moreover, as before we may restrict to the dyadic region $\langle k_3\rangle\sim N_2$ and $\langle k_2 \rangle\sim N_3$ (so $N_2\gtrsim N_3$), where $N_2$ and $N_3$ are dyadic numbers. It will suffice to bound   
\begin{equation}\label{z5:bound01}
\left\|
\langle k \rangle^{\frac{1}{2}} \sum_{\Delta} |k_1|  \frac{1}{\langle\Delta\rangle^{1-4\delta}}
  \sum_{\substack{(k_1,k_2,k_3)\in\mathbb{X}_*\\(k-k_2)(k-k_3)=\Delta\\|k_2|\sim N_2, |k_3|\sim N_3}} \prod_{j=1}^3|\widetilde{v_j}(k_j, \lambda_j)|
\right\|_{\ell^{p_0}_k}
\end{equation}
where $\langle\Delta\rangle\sim \langle k\rangle\langle k_1\rangle$ and $N_2\sim|k_2|\gtrsim \max (|k|, |k_1|)$ (by Proposition \ref{splitting}). By the standard divisor bound and H\"older's inequality as the proof of Proposition \ref{contnew} we obtain that
 \begin{equation}\label{z5:bound02}
 (\ref{z5:bound01})\lesssim \frac{\|\langle k_1\rangle^{\frac{1}{2}}\widetilde{v_1}\|_{\ell^{p_0}_k}\|\langle k_2\rangle^\frac{1}{2p_0}\widetilde{v_2}\|_{\ell^{r_2}_k}\|\langle k_3\rangle^{\frac{1}{2p_0}}\widetilde{v_3}\|_{\ell^{r_2}_k}}{N_2^{\frac{1}{2p_0}-2\delta-\theta} N_3^{\frac{1}{2p_0}}}. 
 \end{equation}
Then we may integrate over $\lambda_1$ $\lambda_2$ and $\lambda_3$ and sum over $(N_2, N_3)$.  By using the negative power of $N_2$ (suppose $\delta< {1}/{(4p_0)}$) and the following facts (similar as before):
\begin{equation}\label{z5:normalize}
\|\langle k_1\rangle^{\frac{1}{2}}\widetilde{v_1}\|_{L_\lambda^1 \ell^{p_0}_k}\lesssim \|v_1\|_{Z_0},\quad
\|\langle k_2\rangle^\frac{1}{2p_0}\widetilde{v_2}\|_{L_\lambda^1 \ell^{r_2}_k}\lesssim \|v_2\|_{Y_0},\quad
\|\langle k_3\rangle^\frac{1}{2p_0}\widetilde{v_3}\|_{L_\lambda^1\ell^{r_2}_k}\lesssim \|v_3\|_{Y_0},
\end{equation}
this finishes the proof of (\ref{z5:X0}).

Before we start to prove the parts (2) and (3), we may first hold a easier bound for $|K_\Delta^{X,+}|$.
Suppose $|\lambda_j|=\max_{1\leq \ell\leq 3}$ and $|\lambda_j|\gtrsim |\sigma-\Delta|$. Recall that $b_0 = 1-2\delta$ and $b_1 = 1-\delta$, and then we obtain that 
\begin{equation}\label{z5:KX+}
\frac{\langle\lambda\rangle^{b_1}}{\langle\lambda_j\rangle^{b_0-\delta}}|K_\Delta^{X,+}|\lesssim \frac{1}{\langle \Delta\rangle^{1-6\delta} \langle\sigma-\lambda\rangle}.
\end{equation}
We may then fix the other two $\lambda_\ell$ ($\ell\neq j$) and $\Delta$, and then we integrate over $\lambda_j$ and $\lambda$. We can obtained the following bound:
\begin{equation}\label{z5:bound11}
\left\|
\int_{\lambda_j} \frac{1}{\langle\lambda_1+\lambda-\lambda_2-\lambda_3-\Delta\rangle} 
\left(\langle\lambda_j\rangle^{b_1} \widetilde{v_j}(k_j, \lambda_j)\right)
\right\|_{L_\lambda^{q_0}}\lesssim \left\|\langle\lambda_j\rangle^{b_0-\delta} \widetilde{v_j}(k_j, \lambda_j)\right\|_{L^{q_1}_{\lambda_j}}
\end{equation}
by Schur's inequalities.  For $|\lambda_j| =\max_{\ell\in\{1,2,3\}} |\lambda_\ell|$, to prove the parts (2) and (3), it will suffice to consider the norm 
\begin{equation}\label{z5:norm1}
\left\|
\langle k \rangle^{\frac{1}{2}}  \sum_{\mathbb{X}_*}  |k_1| \frac{1}{\langle\Delta\rangle^{1-6\delta}} \left\|\langle\lambda_j\rangle^{b_0-\delta} \widetilde{v_j}(k_j, \lambda_j)\right\|_{L^{q_1}_{\lambda_j}}
 \prod_{\ell\in\{1, 2, 3\}-\{j\}}   \|\widetilde{v_\ell}(k_\ell, \lambda_\ell)\|_{L^1_{\lambda_\ell}}
\right\|_{\ell^{p_0}_k }
\end{equation}
By (\ref{z5:KX+}) and (\ref{z5:bound11}).

(2) Let's consider the case when $j = 1$. By (\ref{z5:norm1}), it will suffice to bound 
\begin{equation}\label{z5:bound21}
\left\|
\langle k \rangle^{\frac{1}{2}}  \sum_{\mathbb{X}_*}  \frac{\langle k_1\rangle^{\frac{1}{2}}}{\langle k_2\rangle^{\frac{1}{2p_0}}\langle k_3\rangle^{\frac{1}{2p_0}}} \frac{1}{\langle \Delta\rangle^{1-6\delta}}  \prod_{\ell=1}^3 |{f_\ell}(k_\ell)|
\right\|_{\ell^{p_0}_k},
\end{equation}
where $f_1 (k_1) = \langle k_1\rangle^{\frac{1}{2}}\|\langle\lambda_1\rangle^{b_0-\delta}\widetilde{v_1}\|_{L^{q_1}_\lambda}$ and $f_\ell (k_\ell) = \langle k_\ell\rangle^{\frac{1}{2p_0}} \|\widetilde{v_\ell}\|_{L_\lambda^1}$ for $\ell = 2, 3$. Similar to the proof of Proposition {\ref{contnew}}, we also have similar bounds:
\[
\|f_1\|_{\ell^{p_0}_k}\lesssim \|v_1\|_{Z_0}, \quad \|f_\ell\|_{\ell^{r_2}_k}\lesssim \|v_\ell\|_{Y_0}
\]
for $\ell = 2, 3$.
We may use dyadic decomposition on $(k_2, k_3)$ and sum over $(k_1, k_2, k_3)$ that corresponds to $\Delta$  and  then over $\Delta$ and $k$. Following the same proof as in the part (1), the negative power of $N_2$ help us bound (\ref{z5:bound21}) by $\|v_1\|_{Z_0}\|v_2\|_{Y_0}\|v_3\|_{Y_0}$, when $\delta<1/(6p_0)$.  This finish the proof of (\ref{z5:X1}).

(3) Let's consider the case when $j \in \{2, 3\}$ and denote $i$ is the other number in $\{2, 3\}$. Similarly by (\ref{z5:norm1}), it will suffice to bound 
\begin{equation}\label{z5:bound31}
\left\|
\langle k \rangle^{\frac{1}{2}}  \sum_{\mathbb{X}_*}  \frac{\langle k_1\rangle^{\frac{1}{2}}}{\langle k_2\rangle^{\frac{1}{2p_0}}\langle k_3\rangle^{\frac{1}{2p_0}}} \frac{1}{\langle \Delta\rangle^{1-6\delta}}  \prod_{\ell=1}^3 |{f_\ell}(k_\ell)|
\right\|_{\ell^{p_0}_k},
\end{equation}
where $f_1 (k_1) = \langle k_1\rangle^{\frac{1}{2}}\|\widetilde{v_1}\|_{L^{1}_\lambda}$, $f_j (k_j) = \langle k_j\rangle^{\frac{1}{2p_0}} \|\langle\lambda_j\rangle^{b_0-\delta}\widetilde{v_j}\|_{L_\lambda^{q_1}}$ and $f_i (k_i) = \langle k_i\rangle^{\frac{1}{2p_0}} \|\widetilde{v_i}\|_{L_\lambda^1}$. Similar to the proof of Proposition {\ref{contnew}}, we also have similar bounds:
\[
\|f_\ell\|_{\ell_k^{p_0}}\lesssim \|v_\ell\|_{Z_0}, \quad \|f_i\|_{\ell_k^{r_2}}\lesssim \|v_i\|_{Y_0}
\]
for $\ell =1, j$. Following the same proof of the part (2), (\ref{z5:bound31}) can be bounded by $\|v_1\|_{Z_0}\|v_j\|_{Z_0}\|v_i\|_{Y_0}$ when $\delta<1/(6p_0)$. This finishes the proof of (\ref{z5:X23}).\end{proof}
\subsection{The canonical quintic term} The majority of $z_{j\ell}$ are quintic terms; in fact the majority of them can be treated in the same way, using the following estimate.
\begin{prop}
\label{canonical} Consider a quintic expression $\mathcal{R}$ that satisfies
\begin{equation}\label{defr}
|\mathscr{X}\mathcal{R}(v_1,\cdots,v_5)(k,\lambda)|\lesssim\sum_{\pm k_1\pm\cdots\pm k_5=k}\int_{\mathbb{R}}\frac{\mathrm{d}\sigma}{\langle\lambda\rangle^{1-\theta}\langle\lambda-\sigma\rangle^{1-\theta}}\int_{\pm\lambda_1\pm\lambda_2\pm\cdots\pm\lambda_5=\sigma-\Xi}\prod_{j=1}^5|\widetilde{v_j}(k_j,\lambda_j)|,
\end{equation} where $\Xi:=k^2\mp k_1^2\mp\cdots \mp k_5^2$ (the signs are arbitrary, but the signs of $\pm k_j$ and $\mp k_j^2$ are always the opposite). Then after the operation in Section \ref{splitz}, step (6), the resulting terms satisfy the corresponding multilinear estimates. In particular, suppose the chosen index during this operation is $i=1$, then we have
\begin{align}\label{quincbdd}\|\mathcal{R}(v_1,\cdots,v_5)\|_{Z_1}&\lesssim \|v_1\|_{Z_0}\prod_{j=2}^5\|v_j\|_{Y_0},\\
\label{quincbdd1}\|\mathcal{R}(\varphi_T\cdot\mathcal{E}_N^Y(v_1,v_2,v_3),v_4,\cdots,v_7)\|_{Z_1}&\lesssim \|v_1\|_{Z_0}\|v_2\|_{Z_0}\prod_{j=3}^7\|v_j\|_{Y_0},\\
\label{quincbdd2}\|\mathcal{R}(\varphi_T\cdot\mathcal{E}_L^Y(v_1,v_2,v_3),v_4,\cdots,v_7)\|_{Z_1}&\lesssim \|v_1\|_{Z_0}\prod_{j=2}^7\|v_j\|_{Y_0}.
\end{align}
\end{prop}
\begin{proof} (1) We first prove (\ref{quincbdd}). Assume all the norms on the right hand side are $1$. By a dyadic decomposition, we may restrict to the region where $\langle k_j\rangle\sim N_j$ for $2\leq j\leq 5$; by symmetry we may assume $N_2\geq\cdots\geq N_5$. As in the proof of Proposition \ref{contnew} we have
\[\|\langle k_1\rangle^{\frac{1}{2}}\widetilde{v_1}\|_{L_{\lambda}^1\ell_k^{p_0}}\lesssim 1\] and 
\[\|\langle k_j\rangle^{(1-\sqrt{\delta})/p_0}\widetilde{v_j}\|_{L_\lambda^1\ell_k^{2}}\lesssim\|\langle k_j\rangle^{(1-\sqrt{\delta})/p_0}\widetilde{v_j}\|_{L_\lambda^1\ell_k^{r_2}}\lesssim 1\] for $2\leq j\leq 5$. We may then again fix $\lambda_j$ for $1\leq j\leq 5$, which we eventually integrate over, and assume
\[f_1(k_1)=\langle k_1\rangle^{\frac{1}{2}}|\widetilde{v_1}(k_1,\lambda_1)|,\quad f_j(k_j)=N_j^{(1-\sqrt{\delta})/p_0}|\widetilde{v_j}(k_j,\lambda_j)|\,\,(2\leq j\leq 5),\] such that (after a further normalization)
\begin{equation}\label{quinfjbd}\|f_1\|_{\ell_k^{p_0}}\leq1,\quad \|f_j\|_{\ell_k^2}\lesssim 1\,\,(2\leq j\leq 5).
\end{equation} Using also that $\langle k\rangle\lesssim\langle k_1\rangle$, it then suffices to prove that
\begin{equation}\label{goal0}\bigg\|\langle\lambda\rangle^{b_1}\sum_{\pm k_1\pm\cdots\pm k_5=k}\int_{\mathbb{R}}\frac{1}{\langle\lambda\rangle^{1-\theta}\langle\lambda-\Xi-\mu\rangle^{1-\theta}}\prod_{j=1}^5f_j(k_j)\bigg\|_{\ell_k^{p_0}L_\lambda^{q_0}}\lesssim N_2^{-\theta}(N_2N_3N_4N_5)^{(1-\sqrt{\delta})/p_0}\end{equation} for any fixed $\mu$ (which is a linear combination of $\lambda_j$ for $1\leq j\leq 5$). Using the fact that $b_1<1-\theta$ and Schur's estimate, we can bound for fixed $k$ that
\[\bigg\|\langle\lambda\rangle^{b_1}\sum_{\pm k_1\pm\cdots\pm k_5=k}\int_{\mathbb{R}}\frac{1}{\langle\lambda\rangle^{1-\theta}\langle\lambda-\Xi-\mu\rangle^{1-\theta}}\prod_{j=1}^5f_j(k_j)\bigg\|_{L_\lambda^{q_0}}\lesssim\|F(k,\Xi)\|_{\ell_\Xi^{q_1}},
\] where
\[F(k,\Xi)=\sum_{\substack{\pm k_1\cdots\pm k_5=k\\\pm k_1^2\cdots\pm k_5^2=\Xi-k^2}}\prod_{j=1}^5f_j(k_j).\] As $\Xi$ is determined by $(k_2,k_3,k_4,k_5)$ and hence the number of different $\Xi$'s does not exceed $O(N_2^4)$, we can bound the  $\ell_\Xi^{q_1}$ norm of $F(k,\Xi)$ by $N_2^{\frac{4}{q_1}}$ times its $\ell_\Xi^{\infty}$ norm.

(a) Assume $N_3\geq N_2^{4\sqrt{\delta}}$. For fixed $k$ and $\Xi$, by assumption we know that either there is no pairing, or there is a pairing, say $(2,3)$, and there is no pairing in $\{1,4,5\}$, or there are two pairings, say $(2,3)$ and $(4,5)$. In the first case, for fixed $(k_2,k_3)$ (or $(k_4,k_5)$), the number of choices for $(k_1,k_4,k_5)$ (or $(k_1,k_2,k_3)$) is at most $O(N_2^\theta)$ by Lemma \ref{newdiv}, so we can bound using (\ref{quinfjbd}) that
\[F(k,\Xi)\lesssim \sum_{k_2,k_3}\sum_{k_1,k_4,k_5}f_1(k_1)f_2(k_2)^2f_3(k_3)^2+\sum_{k_4,k_5}\sum_{k_1,k_2,k_3}f_1(k_1)f_4(k_4)^2f_5(k_5)^2\lesssim N_2^\theta\sup_{|k_1-k|\lesssim N_2}f_1(k_1).\] In the other two cases this estimate can be similarly established. This gives
\[\|F(k,\Xi)\|_{\ell_\Xi^{\infty}}\lesssim N_2^\theta \bigg(\sum_{|k_1-k|\lesssim N_2}|f_1(k_1)|^{p_0}\bigg)^{\frac{1}{p_0}},\] and hence the left hand side of (\ref{goal0}) is bounded by
\[N_2^{\theta+ \frac{4}{q_1}}\bigg\|\bigg(\sum_{|k_1-k|\lesssim N_2}|f_1(k_1)|^{p_0}\bigg)^{\frac{1}{p_0}}\bigg\|_{\ell_k^{p_0}}\lesssim N_2^{\theta+ \frac{4}{q_1}+ \frac{1}{p_0}}\] using (\ref{quinfjbd}). As $N_3\geq N_2^{4\sqrt{\delta}}$, $4/q_1 =O(\delta)$, $\delta$ is small enough depending on $p_0$ and $\theta$ is small enough depending on $\delta$, this implies (\ref{goal0}).

(b) Assume $N_3\leq N_2^{4\sqrt{\delta}}$. Then in estimating $F(k,\Xi)$, we may fix the choices of $(k_3,k_4,k_5)$ and eventually sum over them. In this process we lost at most $N_2^{O(\sqrt{\delta})}$. Then, with $(k_3,k_4,k_5)$ fixed, we know that $(k_1,k_2)$ is uniquely determined by $k$ and $\Xi$, since by assumption $(1,2)$ cannot be a pairing. Thus, with $(k_3,k_4,k_5)$ fixed, we have
\[|F(k,\Xi)|\lesssim\sup_{k_2}f_1(k-\ell-k_2)f_2(k_2),\] where $\ell$ is a linear combination of $(k_3,k_4,k_5)$. As
\[
\big\|\sup_{k_2}f_1(k-\ell-k_2)f_2(k_2)\big\|_{\ell_k^{p_0}}\lesssim\bigg\|\bigg(\sum_{k_2}f_1(k-\ell-k_2)^{p_0}f_2(k_2)^{p_0}\bigg)^{\frac{1}{p_0}}\bigg\|_{\ell_k^{p_0}}=\|f_1\|_{\ell_k^{p_0}}\|f_2\|_{\ell_k^{p_0}}\] and $\|f_2\|_{\ell_k^{p_0}}\lesssim\|f_2\|_{\ell_k^2}\lesssim 1$, this bounds the left hand side of (\ref{goal0}) by $N_2^{ \frac{4}{q_1} + O(\sqrt{\delta})}$, which also suffices as $\delta$ is small enough depending on $p_0$.

(2) Next we will prove (\ref{quincbdd1}) and (\ref{quincbdd2}). By (\ref{defr}) and (\ref{formforxy})$\sim$(\ref{bdd1}), we can write (where $*\in\{N,L\}$)
\begin{multline}\label{form2}|\widetilde{\mathcal{R}}(k,\lambda)|\lesssim\sum_{\pm k'\pm k_4\pm\cdots\pm k_7=k}\int_{\mathbb{R}}\frac{\mathrm{d}\sigma}{\langle\lambda\rangle^{1-\theta}\langle\lambda-\sigma\rangle^{1-\theta}}\int_{\pm\lambda'\pm\lambda_4\pm\cdots\pm\lambda_7=\sigma-\Xi}\prod_{j=4}^7|\widetilde{v_j}(k_j,\lambda_j)|\\\sum_{\substack{k_2+k_3-k_1=k'\\(k_1,k_2,k_3)\in \mathbb{X}_*}}|k_1|
\int_{\mathbb{R}}T\widehat{\varphi}(T(\lambda'-\mu'))\,\mathrm{d}\mu'\int_{\mathbb{R}}\min\bigg(\frac{1}{\langle \Delta'\rangle},\frac{1}{\langle\mu'\rangle}\bigg)\frac{\mathrm{d}\sigma'}{\langle \mu'-\sigma'\rangle}\int_{\lambda_2+\lambda_3-\lambda_1=\sigma'-\Delta'}\prod_{j=1}^{3}|\widetilde{v_j}(k_j,\lambda_j)|.
\end{multline} Here $\Xi=k^2\mp (k')^2\mp k_4^2\mp\cdots\mp k_7^2$, and $\Delta'=(k')^2+k_1^2-k_2^2-k_3^2$. This can be reduced to
\begin{equation}\label{form3}\sum_{\pm k_1\pm\cdots\pm k_7=k}\frac{|k_1|}{|\Delta'|}\int_{\mathbb{R}}R(\lambda,\tau)\,\mathrm{d}\tau\int_{\pm\lambda_1\pm\cdots\pm\lambda_7=\tau-\Xi'}\prod_{j=1}^7|\widetilde{v_j}(k_j,\lambda_j)|,\end{equation} where $\Xi'=k^2\mp k_1^2\mp\cdots\mp k_7^2$, and the kernel
\[|R(\lambda,\tau)|\lesssim\int_{\mathbb{R}}\frac{\mathrm{d}\sigma}{\langle \lambda\rangle^{1-\theta}\langle \lambda-\sigma\rangle^{1-\theta}}\int_{\mathbb{R}}T\widehat{\varphi}(T(\xi-\sigma))\frac{\mathrm{d}\xi}{\langle\tau-\xi\rangle}.\] Here we can verify that $\lambda'-\sigma'=\tau-\sigma$, and $\xi$ is the variable such that $\xi-\sigma=\lambda'-\mu'$ and $\tau-\xi=\mu'-\sigma'$. Using the fact that $|T\widehat{\varphi}(T\xi)|\lesssim\langle \xi\rangle^{-1}$, we can easily bound the above by
\[|R(\lambda,\tau)|\lesssim\frac{1}{\langle \lambda\rangle^{1-\theta}\langle\lambda-\tau\rangle^{1-2\theta}}.\] We may then restrict to the dyadic region $\langle k\rangle\sim N_0$, $\langle k'\rangle\sim N'$ and $\langle k_j\rangle\sim N_j$ for $1\leq j\leq 7$. Let $N^+$ be the maximum of all the $N_j$'s. Then we have $N'\gtrsim N_0$, and $|\Delta'|\sim N'N_1$.

(a) Assume $N_1\lesssim N'$, we will then measure $v_2$ in the $Y_0$ norm. By repeating the above proof and fixing $\lambda_j$ for $1\leq j\leq 7$, we may reduce to proving
\begin{equation}\label{reduce}\|F(k,\Xi')\|_{\ell_k^{p_0}\ell_{\Xi'}^\infty}\lesssim N_0^{-\frac{1}{2}}N'N_1^{\frac{1}{2}}\cdot(N_2\cdots N_7)^{(1-\sqrt{\delta})/p_0-(1/q_1)}(N^+)^{-\theta},\end{equation} where
\begin{equation}\label{deffkxi'}F(k,\Xi')=\sum_{\substack{\pm k_1\pm\cdots\pm k_7=k\\\pm k_1^2\pm\cdots\pm k_7^2=k^2-\Xi'}}\prod_{j=1}^7f_j(k_j),\end{equation} and 
\begin{equation}\label{sepbd}\|f_1\|_{\ell_k^{p_0}}\leq 1,\quad \|f_j\|_{\ell_k^2}\leq 1\,\,(2\leq j\leq 7).
\end{equation} First assume $\max(N_2,\cdots, N_7)\geq (N^+)^{\sqrt{\delta}}$, then with fixed $k$, $k_1$ and $\Xi$, by using Lemma \ref{newdiv} and similar arguments as in the above proof, we can easily show that (whether or not there is any pairing in $\{2,3,\cdots,7\}$)
\begin{equation}\label{sextic}\sum_{\substack{\pm k_2\pm\cdots\pm k_7=\rm{const.}\\\pm k_2^2\pm\cdots\pm k_7^2=\rm{const.}}}\prod_{j=2}^7f_j(k_j)\lesssim\max(N_2,\cdots,N_7)^{\theta}\prod_{j=2}^7\|f_j\|_{\ell_k^2},
\end{equation} therefore
\[|F(k,\Xi')|\lesssim \max(N_2,\cdots,N_7)^{\theta}\sum_{k_1}f_1(k_1)\lesssim\max(N_2,\cdots,N_7)^{\theta}N_1^{1-(\frac{1}{p_0})}\] pointwise in $(k,\Xi')$, thus\[\|F(k,\Xi')\|_{\ell_k^{p_0}\ell_{\Xi'}^\infty}\lesssim N_0^{\frac{1}{p_0}}N_1^{1-(\frac{1}{p_0})}.\] Using the fact that $N'\gtrsim\max(N_0,N_1)$ and $\max(N_2,\cdots, N_7)\geq (N^+)^{\sqrt{\delta}}$, this easily implies (\ref{reduce}). Next assume $\max(N_2,\cdots, N_7)\leq (N^+)^{\sqrt{\delta}}$, then $N^+\sim N'$, and by fixing $(k_2,\cdots,k_7)$ we easily deduce that $|F(k,\Xi')|\lesssim (N')^{O(\sqrt{\delta})}$, from which (\ref{reduce}) follows trivially.

(b) Assume $N_1\gg N'$, then we must have $*=N$ and $N_1\sim N_2\gg N'\gtrsim N_0$. In particular $v_2$ will be measured in the $Z_0$ norm, so we may reduce to proving
\begin{equation}\label{reduce2}\|F(k,\Xi')\|_{\ell_k^{p_0}\ell_{\Xi'}^\infty}\lesssim N_0^{-\frac{1}{2}}N'N_1\cdot(N_3\cdots N_7)^{(1-\sqrt{\delta})/p_0}(N^+)^{-\theta-(6/q_1)},\end{equation} where $F(k,\Xi')$ is as (\ref{deffkxi'}), and
\begin{equation}\label{sepbd2}\|f_1\|_{\ell_k^{p_0}}\leq 1,\quad \|f_2\|_{\ell_k^{p_0}}\leq 1,\quad \|f_j\|_{\ell_k^2}\leq 1\,\,(3\leq j\leq 7).
\end{equation} Here we argue in the same way as case (1), using (\ref{sextic}), but make the additional observation that for fixed $k_1$ we must have $|k_2-k_1|\sim N'$. Therefore
\[\sum_{\substack{\pm k_2\pm\cdots\pm k_7=\rm{const.}\\\pm k_2^2\pm\cdots\pm k_7^2=\rm{const.}}}\prod_{j=2}^7f_j(k_j)\lesssim(N^+)^{\theta}\prod_{j=3}^7\|f_j\|_{\ell_k^2}\cdot\|f_2\cdot\mathbf{1}_{|k_2-k_1|\sim N'}\|_{\ell_k^2}\lesssim (N^+)^{\theta}(N')^{(\frac{1}{2})-(\frac{1}{p_0})},\] and hence
\[\|F(k,\Xi')\|_{\ell_k^{p_0}\ell_{\Xi'}^\infty}\lesssim N_0^{\frac{1}{p_0}}N_1^{1-(\frac{1}{p_0})}(N')^{(\frac{1}{2})-(\frac{1}{p_0})}(N^+)^\theta.\] Using the fact that $N_0\lesssim N'$, this implies (\ref{reduce2}).
\end{proof}
 \begin{rem}\label{gain} From the proof above we actually deduce something slightly stronger: the bounds (\ref{quincbdd})$\sim$(\ref{quincbdd2}) remain true if the right hand side of (\ref{defr}) gets multiplied by $\langle k^+\rangle^\theta$ where $|k^+|$ is the maximum of all relevant frequencies, unless $\langle k\rangle \sim\langle k_1\rangle\gtrsim N_2^{100}$. This fact will be used in the analysis of the $z_{5*}$ terms of Section \ref{remain}.
 \end{rem}
To apply Proposition \ref{canonical}, we will verify that $z_{11}$, $z_{2\ell} \, (2\leq \ell\leq 5)$, $z_{3\ell} \,  (1\leq\ell\leq 4)$ and $z_{4\ell}  \, (1\leq \ell\leq 2)$ all have the form (\ref{defr}). The claim for $z_{11}$ follows from (\ref{esti}). For the other terms, let us look at $z_{22}$ as an example. By (\ref{esti}) and (\ref{formforxy}) we have, for $z_{22}=z_{22}(v_1,\cdots,v_5)$, that
\begin{multline}\label{z22exp}|\widetilde{z_{22}}(k,\lambda)|\lesssim_B\sum_{\substack{k_4+k_5-k'=k\\(k',k_4,k_5)\in\mathbb{X}_H\cup \mathbb{X}_S}}|k'|\int_{\mathbb{R}}\bigg(\frac{1}{\langle \lambda\rangle^B}+\frac{1}{\langle \lambda-\tau\rangle^B}\bigg)\frac{\mathrm{d}\tau}{\langle\tau\rangle}\int_{\lambda_4+\lambda_5-\lambda'=\tau-\Delta}|\widetilde{v_4}(k_4,\lambda_4)||\widetilde{v_5}(k_5,\lambda_5)|\\
\sum_{\substack{k_2+k_3-k_1=k'\\(k_1,k_2,k_3)\in \mathbb{X}_{N}}}|k_1|\int_{\mathbb{R}}T\widehat{\varphi}(T(\lambda'-\mu'))\,\mathrm{d}\mu'\int_{\mathbb{R}}\min\bigg(\frac{1}{\langle\Delta'\rangle},\frac{1}{\langle \mu\rangle}\bigg)\frac{1}{\langle\mu'-\sigma'\rangle}\,\mathrm{d}\sigma'\int_{\lambda_2+\lambda_3-\lambda_1=\sigma'-\Delta'}\prod_{j=1}^3|\widetilde{v_j}(k_j,\lambda_j)|,
\end{multline} where $\Delta=k^2+(k')^2-k_4^2-k_5^2$ and $\Delta'=(k')^2+k_1^2-k_2^2-k_3^2$. Note that $|\Delta'|\sim \langle k\rangle\langle k_1\rangle$, the above can be written as
\[\sum_{k_1-k_2-k_3+k_4+k_5=k}\int_{\mathbb{R}}R(\lambda,\sigma)\,\mathrm{d}\sigma\int_{\lambda_1-\lambda_2-\lambda_3+\lambda_4+\lambda_5=\sigma-\Xi}\prod_{j=1}^5|\widetilde{v_j}(k_j,\lambda_j)|,\] where $\Xi=\Delta-\Delta'=k^2-k_1^2+k_2^3+k_3^2-k_4^2-k_5^2$, and 
\[|R(\lambda,\sigma)|\lesssim\int_{\mathbb{R}}\bigg(\frac{1}{\langle \lambda\rangle^B}+\frac{1}{\langle \lambda-\tau\rangle^B}\bigg)\frac{\mathrm{d}\tau}{\langle\tau\rangle}\int_{\mathbb{R}}T\widehat{\varphi}(T(\sigma-\xi)).\frac{\mathrm{d}\xi}{\langle\xi-\tau\rangle}\] Here we can verify that $\lambda'-\sigma'=\sigma-\tau$, and $\xi$ is the variable such that $\sigma-\xi=\lambda'-\mu'$ and $\xi-\tau=\mu'-\sigma'$. The above integral can easily be bounded by $\langle\lambda\rangle^{-1+\theta}\langle\lambda-\sigma\rangle^{-1+\theta}$, so Proposition \ref{canonical}
 can be applied.
 
 The other $z_{2\ell}$, $z_{3\ell}$ and $z_{4\ell}$ terms can be treated in the same way; in fact the kernel $R(\lambda,\sigma)$ will have exactly the same form, the only difference is that the weight
 \[\frac{|k|\cdot|k_1|}{|\Delta'|}\] will be replaced by different weights depending on which input function gets substituted by $\mathcal{E}^Y$, and which $\mathbb{X}_*$ subset we are in. For the terms $z_{22}$, $z_{23}$, $z_{31}$, $z_{32}$, $z_{41}$ and $z_{42}$ one can directly check that this weight is $\lesssim 1$; for the terms $z_{24}$, $z_{25}$, $z_{33}$ and $z_{34}$, this weight is $\lesssim 1$ as it follows from Proposition \ref{splitting} that $|k_2|\gtrsim |k_1|$ when $(k_1,k_2,k_3)\in\mathbb{X}_H\cup \mathbb{X}_S\cup\mathbb{X}_N$. Thus Proposition \ref{main3} has been proved for these terms.
 \subsection{Remaining quintic terms}\label{remain} The remaining quintic terms, namely $z_{26}$, $z_{27}$ and quintic $z_{5*}$ terms, may not have the canonical form (\ref{defr}). In fact these terms will be estimated directly without preforming the operation in step (6) of Section \ref{splitz}, see Remark \ref{extrarem}. For them we need two extra estimates, stated in the following two propositions.
 \begin{prop}\label{extra1}Suppose a quintic term $\mathcal{R}$ satisfies
 \begin{equation}\label{defr2}|\widetilde{\mathcal{R}}(k,\lambda)|\lesssim\sum_{\pm k_1\pm\cdots\pm k_5=k}\alpha(k,k_1,\cdots,k_5)\int_{\mathbb{R}}\frac{1}{\langle\lambda\rangle^{1-\theta}\langle\lambda-\sigma\rangle^{1-\theta}}\int_{\pm\lambda_1\pm\cdots\pm\lambda_5=\sigma-\Xi}\prod_{j=1}^5|\widetilde{v_j}(k_j,\lambda_j)|,
 \end{equation} where as usual $\Xi=k^2\mp k_1^2\mp\cdots\mp k_5^2$. Then we have the following (below $|k^+|$ will denote the maximum of all relevant frequencies):
 
 (1) Assume $|k_1|/2\leq |k_2|\leq 2|k_1|$, $|k_1|\geq 2^{5}|k_3|$, $|k_3|\sim\max(|k_3|,|k_4|,|k_5|)$ and \begin{equation}|\alpha|\lesssim\langle k^+\rangle^\theta\frac{\langle k_1\rangle}{\max(\langle k_3\rangle,\langle k_4\rangle,\langle k_5\rangle,\langle k\rangle)},\end{equation} moreover assume $(1,2)$ is not a pairing. Then we have that
 \begin{equation}\label{extest1}\|\mathcal{R}\|_{Z_1}\lesssim\|v_1\|_{Z_0}\|v_2\|_{Z_0}\|v_3\|_{Z_0}\cdot\|v_4\|_{Y_0}\|v_5\|_{Y_0};
 \end{equation}
 
 (2) Assume $|k_1|/2\leq |k_2|\leq 2|k_1|$, $|k_3|/2\leq|k_4|\leq 2|k_3|$, $|k_1|\geq 2^5\max(|k_5|,|k|)$ and 
 \begin{equation}|\alpha|\lesssim\langle k^+\rangle^\theta\frac{\langle k_1\rangle}{\max(\langle k_5\rangle,\langle k\rangle)}.\end{equation} Moreover assume $(1,2)$ is not a pairing, and that, either $|k|\neq |k_5|$, or the stronger bound
 \begin{equation}\label{stronger0} |\alpha|\lesssim\langle k^+\rangle^\theta\frac{\langle k_1\rangle}{\max(\langle k_5\rangle,\langle k\rangle,\langle\pm k_1\pm k_2\rangle)}
 \end{equation} holds. Then we have
 \begin{equation}\label{extest2}\|\mathcal{R}\|_{Z_1}\lesssim\|v_1\|_{Z_0}\|v_2\|_{Z_0}\|v_3\|_{Z_0}\|v_4\|_{Z_0}\cdot\|v_5\|_{Y_0}.
 \end{equation}
 \begin{proof} We may restrict to the region where $\langle k_j\rangle\sim N_j$, where $1\leq j\leq 5$, and $\langle k\rangle\sim N_0$ and $\langle k^+\rangle\sim N^+$.
 
 (1) By the same arguments as in the proof of Propositions \ref{contnew} and \ref{canonical}, we may fix $\lambda_j(1\leq j\leq 5)$ and reduce to estimating
 \begin{equation}\label{goal10}\bigg\|\langle\lambda\rangle^{b_1}\sum_{\pm k_1\pm\cdots\pm k_5=k}\int_{\mathbb{R}}\frac{1}{\langle\lambda\rangle^{1-\theta}\langle\lambda-\Xi-\mu\rangle^{1-\theta}}\prod_{j=1}^5f_j(k_j)\bigg\|_{\ell_k^{p_0}L_\lambda^{q_0}}\lesssim (N^+)^{-2\theta}\frac{N'N_3^{\frac{1}{2}}(N_4N_5)^{(1-\sqrt{\delta})/p_0}}{N_0^{\frac{1}{2}}},\end{equation} where $N'=\max(N_0,N_3)$, and $f_j$ satisfies that
 \begin{equation}\label{bdfj01}\|f_j\|_{\ell_k^{p_0}}\leq 1,\quad 1\leq j\leq 3;\qquad \|f_j\|_{\ell_k^{2}}\leq 1, \quad 4\leq j\leq 5.
 \end{equation} By the same argument as in the proof of Proposition \ref{canonical}, we may apply Schur's estimate and reduce to proving
 \begin{equation}\label{goal11}\|F(k,\Xi)\|_{\ell_k^{p_0}\ell_{\Xi}^{q_1}}\lesssim (N^+)^{-2\theta}\frac{N'N_3^{\frac{1}{2}}(N_4N_5)^{(1-\sqrt{\delta})/p_0}}{N_0^{\frac{1}{2}}},\quad F(k,\Xi):=\sum_{\substack{\pm k_1\pm\cdots\pm k_5=k\\\pm k_1^2\pm\cdots\pm k_5^2=k^2-\Xi}}\prod_{j=1}^5f_j(k_j).\end{equation} By fixing $(k_4,k_5)$ we get that
 \[\|F(k,\Xi)\|_{\ell_{\Xi}^{q_1}}\lesssim\|f_4\|_{\ell_k^1}\|f_5\|_{\ell_k^1}\sup_{\ell,\rho}\|F_{\ell,\rho}(k,\Xi)\|_{\ell_k^{p_0}\ell_{\Xi}^{q_1}},\quad F_{\ell,\rho}(k,\Xi):=\sum_{\substack{\pm k_1\pm k_2\pm k_3=k+\ell\\\pm k_1^2\pm k_2^2\pm k_3^2=k^2-\Xi+\rho}}\prod_{j=1}^3f_j(k_j),\] while since there is no pairing in $\{1,2,3\}$, by the standard divisor estimate we have
 \[|F_{\ell,\rho}(k,\Xi)|\lesssim (N^+)^{\theta}\bigg(\sum_{\substack{\pm k_1\pm k_2\pm k_3=k+\ell\\\pm k_1^2\pm k_2^2\pm k_3^2=k^2-\Xi+\rho}}\prod_{j=1}^3f_j(k_j)^{p_0}\bigg)^{\frac{1}{p_0}},\] and hence $\|F_{\ell,\rho}(k,\Xi)\|_{\ell_k^{p_0}\ell_{\Xi}^{q_1}}\lesssim \|F_{\ell,\rho}(k,\Xi)\|_{\ell_k^{p_0}\ell_{\Xi}^{p_0}}\lesssim (N^+)^{\theta}$. Using also H\"{o}lder we obtain \[\|F(k,\Xi)\|_{\ell_k^{p_0}\ell_{\Xi}^{q_1}}\lesssim (N_4N_5)^{\frac{1}{2}}(N^+)^\theta.\]
 
 Comparing with (\ref{goal11}) and using that $N'\sim\max(N_0,N_3)$ and $\max(N_4,N_5)\lesssim N_3$, we see that (\ref{goal11}) is proved, except for the loss $(N^+)^\theta$. Clearly this loss can be covered if $N'\gtrsim N_1^{1/10}$; now suppose $\max(N_0,N_3,N_4,N_5)\ll N_1^{1/10}$, then since $(1,2)$ is not a pairing, we must have $|\Xi|\gtrsim N^+$, which gives
 \[\max(|\lambda_1|,\cdots,|\lambda_5|,|\lambda|,|\lambda-\Xi-\mu|)\gtrsim N^+,\] where $\mu$ is a linear combination of $\lambda_1,\cdots\lambda_5$. Now, in estimating (\ref{goal10}) we can gain a power $\langle\lambda\rangle^{(1-\theta)-b_1}\geq\langle\lambda\rangle^{\delta/2}$; in the process of fixing $\lambda_j$ we can also gain a power $\langle\lambda_j\rangle^{\delta/2}$, as
 \begin{multline*}\|\langle\lambda_j\rangle^{\delta/2}\langle k_j\rangle^{(1-\sqrt{\delta})/p_0}\widetilde{v_j}\|_{L_\lambda^1\ell_k^{2}}\lesssim\|\langle\lambda_j\rangle^{\delta/2}\langle k_j\rangle^{(1-\sqrt{\delta})/p_0}\widetilde{v_j}\|_{L_\lambda^1\ell_k^{r_2}}\lesssim \|\langle k_j\rangle^{(1-\sqrt{\delta})/p_0}\langle\lambda_j\rangle^{\frac{1}{2}}\widetilde{v_j}\|_{L_\lambda^{r_0}\ell_k^{r_2}}\\\lesssim\|\langle k_j\rangle^{(1-\sqrt{\delta})/p_0}\langle\lambda_j\rangle^{\frac{1}{2}}\widetilde{v_j}\|_{\ell_k^{r_2}L_\lambda^{r_0}}\lesssim\|\langle k_j\rangle^{\frac{1}{2}}\langle \lambda_j\rangle^{\frac{1}{2}}\widetilde{v_j}\|_{\ell_k^{p_0}L_\lambda^{r_0}}.\end{multline*} Finally, in the process of using Schur's estimate to reduce (\ref{goal10}) to (\ref{goal11}), we can also replace the power $\langle\lambda-\Xi-\mu\rangle^{-1+\theta}$ by a slightly larger power gain a power $\langle\lambda-\Xi-\mu\rangle^{\delta/4}$. In this way we can gain a power of at least $(N^+)^{\delta/4}$ which suffices to cover the $(N^+)^{\theta}$ loss.
 
 (2) If there is no pairing in $\{1,2,3\}$, then similar to (1), we may fix $\lambda_j$ and reduce to proving
  \begin{equation}\label{goal12}\|F(k,\Xi)\|_{\ell_k^{p_0}\ell_{\Xi}^{q_1}}\lesssim (N^+)^{-2\theta}\frac{N'N_3\cdot N_5^{(1-\sqrt{\delta})/p_0}}{N_0^{\frac{1}{2}}},\quad F(k,\Xi):=\sum_{\substack{\pm k_1\pm\cdots\pm k_5=k\\\pm k_1^2\pm\cdots\pm k_5^2=k^2-\Xi}}\prod_{j=1}^5f_j(k_j),\end{equation} where $N^+=\max(N_1,N_3)$, $N'=\max(N_0,N_5)$ and
  \begin{equation}\label{fjbd4}\|f_j\|_{\ell_k^{p_0}}\leq 1, \, \, 1\leq j\leq 4;\quad \|f_5\|_{\ell_k^{2}}\leq 1.
  \end{equation}  Then we may fix $k_4$ and $k_5$ and argue as in part (1) to get
  \[\|F(k,\Xi)\|_{\ell_k^{p_0}\ell_{\Xi}^{q_1}}\lesssim (N^+)^\theta\|f_4\|_{\ell_k^1}\|f_5\|_{\ell_k^1}\lesssim (N^+)^\theta N_3^{1-(\frac{1}{p_0})}N_5^{\frac{1}{2}},\] which implies (\ref{goal12}) except for the loss $(N^+)^{\theta}$, which can be covered in the same way as part (1) by considering $\Xi$.
  
  If there is a pairing in $\{1,2,3\}$, say $(1,3)$, then $\frac{1}{2}\leq N_1/N_3\leq 2$. If $(2,4)$ is not a pairing, then we can fix $(k_3,k_5)$ and repeat the above argument to get the same (in fact better) estimate; so we may assume $(2,4)$ is also a pairing. This forces $\Xi=0$ and $k=k_5$, in particular the stronger bound (\ref{stronger0}) holds. Let $\langle \pm k_1\pm k_2\rangle\sim N_6$ and $N''=\max(N',N_6)$. In this case we will still fix $\lambda_j(1\leq j\leq 4)$ but will not fix $\lambda_5$. Instead, let $\mu$ be a linear combination of $\lambda_j(1\leq j\leq 4)$ and is thus fixed, and notice that $\Xi=0$, we have
  \[|\widetilde{\mathcal{R}}(k,\lambda)|\lesssim N_1^{-1}(N'')^{-1}\sum_{\substack{|k_1|\sim|k_2|\sim N_1\\|\pm k_1\pm k_2|\sim N_6}}\int_{\mathbb{R}}\frac{\mathrm{d}\sigma}{\langle\lambda\rangle^{1-\theta}\langle \lambda-\sigma\rangle^{1-\theta}}f_1(k_1)f_2(k_2)f_3(k_1)f_4(k_2)|\widetilde{v_5}(k,\pm\sigma\pm\mu)|,\] where $\|f_j\|_{\ell_k^{p_0}}\leq 1$ for $1\leq j\leq 4$. This implies that
  \[\|\langle \lambda\rangle^{b_1}\widetilde{\mathcal{R}}(k,\lambda)\|_{L_\lambda^{q_0}}\lesssim N_1^{-1}(N'')^{-1}\sum_{\substack{|k_1|\sim|k_2|\sim N_1\\|\pm k_1\pm k_2|\sim N_6}}f_1(k_1)f_2(k_2)f_3(k_1)f_4(k_2)\|\langle\lambda_5\rangle^{\frac{1}{2}}\widetilde{v_5}(k,\lambda_5)\|_{L_\lambda^{r_0}},\] and hence
  \begin{multline*}\|\langle k\rangle^{\frac{1}{2}}\langle \lambda\rangle^{b_1}\widetilde{\mathcal{R}}(k,\lambda)\|_{\ell_k^{p_0}L_\lambda^{q_0}}\\\lesssim N_1^{-1}(N'')^{-1}\|\langle k\rangle^{\frac{1}{2}}\langle\lambda_5\rangle^{\frac{1}{2}}\widetilde{v_5}(k,\lambda_5)\|_{\ell_k^{p_0}L_\lambda^{r_0}}\cdot\sum_{\substack{|k_1|\sim|k_2|\sim N_1\\|\pm k_1\pm k_2|\sim N_6}}f_1(k_1)f_2(k_2)f_3(k_1)f_4(k_2),\end{multline*} while the latter sum is bounded by
  \[\sum_{k_1}f_1(k_1)f_3(k_1)\cdot N_6^{1-(2/p_0)}\|f_2\|_{\ell_k^p}\|f_4\|_{\ell_k^p}\lesssim (N_1N_6)^{1-(2/p_0)},\] which gives the desired estimate as $N''\gtrsim N_6$.
 \end{proof}
 \end{prop}
 \begin{prop}\label{extra2} Suppose a quintic term $\mathcal{R}$ satisfies
 \begin{equation}\label{defr3}|\widetilde{\mathcal{R}}(k,\lambda)|\lesssim\frac{1}{\langle\lambda\rangle^{1+10\delta}}\sum_{\pm k_1\pm\cdots\pm k_5=k}\beta(k,k_1,\cdots,k_5)\int_{\mathbb{R}^5}\prod_{j=1}^5|\widetilde{v_j}(k_j,\lambda_j)|\,\mathrm{d}\lambda_1\cdots\mathrm{d}\lambda_5,
 \end{equation} where (as usual $|k^+|$ is the maximum of all relevant frequencies)
 \begin{equation}|\beta|\lesssim\langle k^+\rangle^{\sqrt{\delta}}\frac{1}{\langle k\rangle\langle\pm k_3\pm k_4\rangle},\qquad \langle k\rangle\gtrsim\langle k_5\rangle.\end{equation} then we have
 \begin{equation}\|\mathcal{R}\|_{Z_1}\lesssim\prod_{j=1}^5\|v_j\|_{Y_0}.\end{equation} Note that all the norms on the right hand side are $Y_0$ (in particular the bound is symmetric in $v_1$ and $v_2$, $v_3$ and $v_4$).
    \end{prop}
 \begin{proof} As before we will restrict to the region where $\langle k_j\rangle \sim N_j$ for $1\leq j\leq 5$, $\langle k\rangle\sim N_0$, and $\langle\pm k_3\pm k_4\rangle\sim N_6$. Let $N^+\sim\langle k^+\rangle$. This time we will \emph{not} fix $\lambda_j$; instead we first integrate in them. We may assume all the norms on the right hand side are $1$. Let
 \[N_j^{\frac{1}{2}}\|\widetilde{v_j}(k_j,\lambda_j)\|_{L_\lambda^1}=f_j(k_j),\] then $\|f_j\|_{\ell_k^{p_0}}\lesssim 1$ as $v_j\in Y_0$. By (\ref{defr3}), it suffices to prove that
\begin{equation}\label{goal21}\bigg\|\langle\lambda\rangle^{b_1}\langle\lambda\rangle^{-1-10\delta}\sum_{\pm k_1\pm\cdots\pm k_5=k}\prod_{j=1}^5f_j(k_j)\bigg\|_{\ell_k^{p_0}L_{\lambda}^{q_0}}\lesssim (N^+)^{-2\sqrt{\delta}}\cdot N_6(N_0N_1N_2N_3N_4N_5)^{\frac{1}{2}},
\end{equation} where
\begin{equation}
\|f_j\|_{\ell_k^{p_0}}\leq 1,\,\,1\leq j\leq 5.
\end{equation} By symmetry we may assume $N_1\leq N_2$ and $N_3\leq N_4$. By the choice of power of $\lambda$, the $L_\lambda^{q_0}$ part is easily estimated, so we only need to bound the $\ell_k^{p_0}$ norm
\[\bigg\|\sum_{\pm k_1\pm\cdots\pm k_5=k}\prod_{j=1}^5f_j(k_j)\bigg\|_{\ell_k^{p_0}}.\] By Young's inequality, this is bounded by the $\ell_k^{p_0}$ of $f_2$ (which is $\sim 1$), multiplied by
\[\sum_{k_1,k_3,k_4,k_5}f_1(k_1)f_3(k_3)f_4(k_4)f_5(k_5).\] The sum over $k_1$ and $k_5$ gives (by H\"{o}lder) $(N_1N_5)^{1-(\frac{1}{p_0})}$; when $k_3$ is fixed the sum over $k_4$ gives $N_6^{1-(\frac{1}{p_0})}$ as $|\pm k_3\pm k_4|\lesssim N_6$, and finally the sum over $k_3$ gives $N_3^{1-(\frac{1}{p_0})}$. This gives the bound
\[(N_1N_3N_5N_6)^{1-(\frac{1}{p_0})}\lesssim (N^+)^{-1/(2p_0)}N_6(N_0N_1N_2N_3N_4N_5)^{\frac{1}{2}},\] as $N_1\leq N_2$, $N_3\leq N_4$ and $N_0\gtrsim N_5$.\end{proof}
Using Propositions \ref{extra1}, we can easily deal with the terms $z_{26}$ and $z_{27}$. For these two terms, by repeating the arguments for $z_{22}$ detailed above, we are led to considering the tuple $(k_1,k_2,k')$ and $(k_3,k_4,k_5)$, such that
\[k_2+k'-k_1=k,\quad (k_1,k_2,k')\in \mathbb{X}_H\cup\mathbb{X}_S;\quad k_3+k_4-k_5=k',\quad (k_3,k_4,k_5)\in \mathbb{X}_N\cup\mathbb{X}_L,\] and a weight
\[\alpha(k,k_1,\cdots,k_5)\sim\frac{|k_1||k_3|}{\langle\Delta'\rangle},\] noticing that $|\Delta'|\geq 1$. By Proposition \ref{quinprep}, this term can be bounded using either Proposition \ref{canonical}, or Proposition \ref{extra1}, (1) or (2).
\subsubsection{The $z_{5*}$ terms} Finally let us consider quintic $z_{5*}$ terms. By (\ref{formforxy})$\sim$(\ref{bdd1}) and (\ref{k+bd}) we write, where $z_{5*}=z_{5*}(v_1,\cdots,v_5)$ and $*,\bullet\in\{N,L\}$, that (strictly speaking $z_{57}$ and $z_{58}$ have a different formula, but taking into account that the set $\mathbb{X}_L$ is symmetric with respect to $k_2$ and $k_3$ - apart from the artificial restriction $|k_2|\geq |k_3|$ - they can be treated in exactly the same way):
 \begin{multline}|\widetilde{z_{5*}}(k,\lambda)|\lesssim_B\sum_{\substack{k_2+k'-k_1=k\\(k_1,k_2,k')\in\mathbb{X}_*}}|k_1|\int_{\mathbb{R}}\bigg[\frac{1}{\langle \lambda\rangle^B\langle\tau\rangle}+\frac{\langle\tau-\Delta\rangle}{\langle\lambda-\tau\rangle^B\langle\tau\rangle}\min\bigg(\frac{1}{\langle\Delta\rangle},\frac{1}{\langle \tau\rangle}\bigg)+\frac{\langle\tau-\Delta\rangle}{\langle\lambda-\tau\rangle}\min\bigg(\frac{1}{\langle\Delta\rangle},\frac{1}{\langle \lambda\rangle}\bigg)^2\bigg]\,\mathrm{d}\tau\\\int_{\lambda_2+\lambda'-\lambda_1=\tau-\Delta}\mathbf{1}_{|\lambda'|\gtrsim|\tau-\Delta|}|\widetilde{v_1}(k_1,\lambda_1)||\widetilde{v_2}(k_2,\lambda_2)|\sum_{\substack{k_4+k_5-k_3=k'\\(k_3,k_4,k_5)\in\mathbb{X}_\bullet}}|k_3|\int_{\mathbb{R}}T\widehat{\varphi}(T(\lambda'-\mu'))\,\mathrm{d}\mu'\\
 \int_{\mathbb{R}}\min\bigg(\frac{1}{\langle\Delta'\rangle},\frac{1}{\langle\mu'\rangle}\bigg)\frac{1}{\langle\mu'-\sigma'\rangle}\,\mathrm{d}\sigma'\int_{\lambda_4+\lambda_5-\lambda_3=\sigma'-\Delta'}|\widetilde{v_3}(k_3,\lambda_3)||\widetilde{v_4}(k_4,\lambda_4)||\widetilde{v_5}(k_5,\lambda_5)|,
 \end{multline} where $\Delta=k^2+k_1^2-k_2^2-(k')^2$ and $\Delta'=(k')^2+k_3^2-k_4^2-k_5^2$. The above can be reduced to
 \begin{equation}\label{reducedx}\sum_{-k_1+k_2-k_3+k_4+k_5=k}\int_{\mathbb{R}}R(\lambda,\sigma)\,\mathrm{d}\sigma\int_{-\lambda_1+\lambda_2-\lambda_3+\lambda_4+\lambda_5=\sigma-\Xi}\prod_{j=1}^5|\widetilde{v_j}(k_j,\lambda_j)|,
 \end{equation} where $\Xi=\Delta+\Delta'=k^2+k_1^2-k_2^2+k_3^2-k_4^2-k_5^2$, and the kernel\footnote{This kernel depends on $k_j$ and $\lambda_j$, but we will write it as $R(\lambda,\sigma)$ for simplicity.}
 \begin{multline}\label{kernelr}
 R(\lambda,\sigma)=\int_{\mathbb{R}}\bigg[\frac{1}{\langle \lambda\rangle^B\langle\tau\rangle}+\frac{\langle\tau-\Delta\rangle}{\langle\lambda-\tau\rangle^B\langle\tau\rangle}\min\bigg(\frac{1}{\langle\Delta\rangle},\frac{1}{\langle \tau\rangle}\bigg)+\frac{\langle\tau-\Delta\rangle}{\langle\lambda-\tau\rangle}\min\bigg(\frac{1}{\langle\Delta\rangle},\frac{1}{\langle \lambda\rangle}\bigg)^2\bigg]\,\mathrm{d}\tau\\
 |k_1k_3|\mathbf{1}_{|\lambda'|\gtrsim|\tau-\Delta|}\int_{\mathbb{R}}T\widehat{\varphi}(T(\lambda'-\mu'))\min\bigg(\frac{1}{\langle\Delta'\rangle},\frac{1}{\langle\mu'\rangle}\bigg)\frac{1}{\langle\mu'-\sigma'\rangle}\,\mathrm{d}\mu'.
 \end{multline} Here $\lambda'=\tau-\Delta+\lambda_1-\lambda_2$ and $\sigma'=\lambda_4+\lambda_5-\lambda_3+\Delta'$ are defined in terms of $\tau$ and $(k_j,\lambda_j)$. First fix $\tau$ and integrate in $\mu'$; this integral is bounded by
 \[\int_{\mathbb{R}}\frac{1}{\langle\lambda'-\mu'\rangle}\min\bigg(\frac{1}{\langle\Delta'\rangle},\frac{1}{\langle\mu'\rangle}\bigg)\frac{1}{\langle\mu'-\sigma'\rangle}\,\mathrm{d}\mu',\] and we separate two cases.
 
 (1) Assume $|\sigma'|\ll|\lambda'|$, then we can calculate that
 \[\frac{1}{\langle \Delta'\rangle}\int_{\mathbb{R}}\frac{1}{\langle\lambda'-\mu'\rangle\langle\mu'-\sigma'\rangle}\,\mathrm{d}\mu'\lesssim\frac{1}{\langle\Delta'\rangle}\frac{1}{\langle\lambda'-\sigma'\rangle^{1-\theta}}\sim\frac{1}{\langle\Delta'\rangle}\frac{1}{\langle\lambda'\rangle^{1-\theta}}.\] Note that $|\lambda'|\gtrsim|\tau-\Delta|$, we can then bound the resulting integral in $\tau$ by
 \begin{multline*}\int_{|\tau-\Delta|\lesssim|\lambda_1-\lambda_2|}\bigg[\frac{1}{\langle \lambda\rangle^B\langle\tau\rangle}+\frac{\langle\tau-\Delta\rangle}{\langle\lambda-\tau\rangle^B\langle\tau\rangle}\min\bigg(\frac{1}{\langle\Delta\rangle},\frac{1}{\langle \tau\rangle}\bigg)\\+\mathbf{1}_{\langle \lambda-\tau\rangle\ll\langle\tau-\Delta\rangle}\frac{\langle\tau-\Delta\rangle}{\langle\lambda-\tau\rangle}\min\bigg(\frac{1}{\langle\Delta\rangle},\frac{1}{\langle \lambda\rangle}\bigg)^2\bigg]\frac{1}{\langle\Delta'\rangle}\frac{1}{\langle\tau-\Delta\rangle^{1-\theta}}\,\mathrm{d}\tau,\end{multline*}
 which is then bounded by
 \[\frac{1}{\langle\Delta\rangle\langle\Delta'\rangle}(\max_j\langle k_j\rangle)^{\sqrt{\delta}}\langle\lambda\rangle^{-1-10\delta}\] by actually performing the integration in $\tau$. By bounding the weight 
 \[\beta=\frac{|k_1||k_3|}{\langle\Delta\rangle\langle\Delta'\rangle}\] using Proposition \ref{quinprep}, we can apply Proposition \ref{extra2} and conclude the estimate for this term.
 
 (2) Assume $|\sigma'|\gtrsim|\lambda'|$, then we can calculate that\[\int_{\mathbb{R}}\frac{1}{\langle\lambda'-\mu'\rangle\langle\mu'\rangle\langle\mu'-\sigma'\rangle}\,\mathrm{d}\mu'\lesssim\frac{1}{\langle\lambda'\rangle^{1-\theta}}\frac{1}{\langle\lambda'-\sigma'\rangle^{1-\theta}}.\] Note that $\lambda'-\sigma'=\tau-\sigma$, and using the fact that $|\lambda'|\gtrsim|\tau-\Delta|$, we can bound the resulting integral in $\tau$ by
 \[\int_{\mathbb{R}}\bigg[\frac{\mathbf{1}_{\langle\sigma\rangle\ll\langle\Delta\rangle}}{\langle \lambda\rangle^B\langle\tau\rangle}+\frac{\langle\tau-\Delta\rangle}{\langle\lambda-\tau\rangle^B\langle\tau\rangle}\min\bigg(\frac{1}{\langle\Delta\rangle},\frac{1}{\langle \tau\rangle}\bigg)+\frac{\langle\tau-\Delta\rangle}{\langle\lambda-\tau\rangle}\min\bigg(\frac{1}{\langle\Delta\rangle},\frac{1}{\langle \lambda\rangle}\bigg)^2\bigg]\frac{1}{\langle\tau-\Delta\rangle^{1-\theta}}\frac{1}{\langle\tau-\sigma\rangle^{1-\theta}}\,\mathrm{d}\tau,\] which can be bounded by
 \[(\max_j\langle k_j\rangle)^{10\theta}\frac{1}{\langle\lambda\rangle^{1-\theta}\langle\lambda-\sigma\rangle^{1-\theta}}\frac{1}{\langle \Delta\rangle}.\] By bounding the weight \[\alpha=\frac{|k_1||k_3|}{\langle\Delta\rangle}\] using Proposition \ref{quinprep}, we can apply either Proposition \ref{canonical}, or Proposition \ref{extra1}, (1) or (2).
 
 In the case we apply Proposition \ref{canonical}, we will also use Remark \ref{gain} to cover the loss $(\max_j\langle k_j\rangle)^{10\theta}$, which can be done unless for some $j$ we have $|k_j|\sim|k|\gtrsim\max_{\ell\neq j}|k_l|^{100}$; in this final case we can check that the stronger bound $|\alpha|\lesssim |k|^{-\frac{1}{2}}$ holds, so the loss can still be covered. This completes the proof of Proposition \ref{main3}.
\section{Preservation of regularity}\label{preg}  Finally in this section we prove a preservation of regularity result. More precisely, 
 we prove the properties of our solution stated in Remark \ref{property}. The following proposition is standard:
\begin{prop}\label{pres} Given $s>\frac{1}{2}$ and $2\leq p_0<\infty$, all the arguments in the previous sections carry over to $H_{p_0}^s$ (and correspondingly $X_{p_0,q_0}^{s,b_j}$ and $X_{p_0,r_j}^{s,\frac{1}{2}}$ for $j\in\{0,1\}$). Moreover, in these arguments $T$ still depends only on the $H_{p_0}^{\frac{1}{2}}$ (instead of $H_{p_0}^s$) size of the initial data.
\end{prop}
\begin{proof} This follows from the elementary inequality that
\[\langle k\rangle^{s-(\frac{1}{2})}\lesssim\max_{1\leq j\leq r}\langle k_j\rangle^{s-(\frac{1}{2})},\quad \mathrm{if\ }k=\pm k_1\pm\cdots\pm k_r,\,\,r\in\{3,5,7\}.\] Thus, any previously proved multilinear estimate will continue to be true if the exponent $\frac{1}{2}$ in the output function space is replaced by $s$, provided that the exponent $\frac{1}{2}$ in \emph{one} appropriate input function space is replaced by $s$. 

Suppose the initial data has $H_{p_0}^{\frac{1}{2}}$ norm $A$ and $H_{p_0}^s$ norm $L$, then for $T=T(A)$, all the $X_{p_0,q_0}^{\frac{1}{2},b_j}(I)$ - and similarly for $X_{p_0,r_j}^{\frac{1}{2},\frac{1}{2}}(I)$ - contraction mappings proved before will still be contraction mapping under the norm
\[\|\cdot\|_{X_{p_0,q_0}^{\frac{1}{2},b_j}(I)}+L^{-1}\|\cdot\|_{X_{p_0,q_0}^{s,b_j}(I)},\quad \mathrm{similarly}\quad \|\cdot\|_{X_{p_0,r_j}^{\frac{1}{2},\frac{1}{2}}(I)}+L^{-1}\|\cdot\|_{X_{p_0,r_j}^{s,\frac{1}{2}}(I)}.\qedhere\]
\end{proof}
Now consider a smooth initial data $u_0$. Proposition \ref{pres} implies that, if $\|u_0\|_{H_{p_0}^{\frac{1}{2}}}\leq A$, then for $T=T(A)$, we can construct a solution to (\ref{dnls}) on $J=[-T,T]$ that belongs to $C_t^0H_{p_0}^s(J)$ for $s$ sufficiently large. This is clearly the classical solution to (\ref{dnls}). If a sequence of smooth initial data $u_0^{(n)}$ satisfies $\|u_0^{(n)}\|_{H_{p_0}^{\frac{1}{2}}}\leq A$ and $u_0^{(n)}\to u_0$ in $H_{p_0}^{\frac{1}{2}}$, then by continuity of the data-to-solution map in $H_{p_0}^{\frac{1}{2}}$ (which follows from the previous proofs), the corresponding solutions $u^{(n)}$ will converge to $u$ in $C_t^0H_{p_0}^{\frac{1}{2}}$, where $u$ is the solution we construct in Theorem \ref{main} with initial data $u_0$. This shows that our solution is the unique limit of smooth solutions.

Finally, suppose $p_0<4$, then the gauged solution $v$ we construct belongs to the space $X_{p_0,r_0}^{\frac{1}{2},\frac{1}{2}}(J)$ where $r_0<2$. It can be shown that
\[X_{p_0,r_0}^{\frac{1}{2},\frac{1}{2}}(J)\subset X_{p_0,2}^{\frac{1}{2},\frac{1}{2}}(J)\cap X_{p_0,1}^{\frac{1}{2},0}(J)\] for any interval $J$ of length not exceeding $1$, so our solution belongs to the function space defined in \cite{GH} in which the authors have proved uniqueness. Therefore when $p_0<4$, our solution must coincide with the one constructed in \cite{GH}, as long as the latter exists.


\begin{thebibliography}{99}

\bibitem{BaiBer} Bailleul, I. and Bernicot, F.,  {\em Heat semigroup and singular PDEs}, J. Funct. Anal. {\bf 270}, no. 9, pp. 3344-3452 [2016].

\bibitem{BiaLin} Biagioni H.A. and Linares, F.,  {\em Ill-posedness for the derivative Schr\"odinger and generalized Benjamin-Ono equations}, Trans. Amer. Math. Soc. {\bf 353}, no. 9, pp. 3649-3659 [2001].

\bibitem{BOP2} B\'enyi,  A., Oh, T. and Pocovnicu, O., {\em On the probabilistic Cauchy theory of the cubic nonlinear Schr\"odinger equation on $\mathbb{R}^d$, $d³3$.} Trans. Amer. Math. Soc. Ser. B 2, pp. 1Ð50 [2015]. 

\bibitem{BOP}  B\'enyi, A., Pocovnicu, O. and Oh, T., {\em Higher order expansions for the probabilistic local Cauchy theory of the cubic nonlinear Schr\"odinger equation on $\mathbb{R}^3$}, Trans. Amer. Math. Soc. Ser. B {\bf 6}, pp. 114-160 [2019].

\bibitem{B5} Bourgain, J., {\em Invariant measures for the $2D$ defocusing nonlinear Schr\"odinger equation}, Comm. Math. Phys.  {\bf 176}, pp. 421-445 [1996].

\bibitem{B6} Bourgain, J., {\em Global solutions of nonlinear Schr\"odinger equations}, Amer. Math. Soc. Colloq.Pub.  {\bf 46}, Amer. Math. Soc. [1999].

\bibitem{B8}  Bourgain, J., {\em Fourier transform restriction phenomena for certain lattice subsets and applications to nonlinear evolution equations. I. Schr\"odinger equations}, GAFA {\bf 3}, pp. 107-156 [1993].

\bibitem{BTz1} Burq, N. and Tzvetkov, N., {\em Random data Cauchy theory for supercritical wave equations.I. Local theory}, Invent. Math. {\bf 173}, no. 3, 449-475 [2008].

\bibitem{BTT} Burq, N. Thoman and L. Tzvetkov, N., {\em Long time dynamics for the one dimensional non linear Schrödinger equation}, Ann. Inst. Fourier (Grenoble) {\bf 63}, no. 6, pp.  2137-2198 [2013].

\bibitem{CCh} Catellier R. and Chouk, K.,  {\em Paracontrolled distributions and the 3-dimensional stochastic quantization equation},  Ann Prob. {\bf 46}, no. 5, pp. 2621--2679  [2018].

\bibitem{ChWe}  Chandra, A. and Weber, H., {\em Stochastic PDEs, regularity structures, and interacting particle systems}, Ann. Fac. Sci. Toulouse Math. (6) {\bf 26}, no. 4, pp. 847-909 [2017].

\bibitem{CCMNS} Chanillo, S., Czubak, M., Mendelson, D., Nahmod, A. and Staffilani, G. {\em Almost sure boundedness of iterates for derivative nonlinear wave equations},
 \href{https://arxiv.org/abs/1710.09346v2}{arXiv:1710.09346}	 [math.AP].
 
 \bibitem{Christ} Christ, M., {\em Power series solution of a nonlinear Schr\"odinger equation}, Mathematical aspects of nonlinear dispersive equations, Ann. of Math. Stud., {\bf 163}, pp.131-155 [2007].

  
\bibitem{CKSTT1} Colliander, J., Keel, M., Staffilani, G., Takaoka, H., and Tao, T.,  
{\em Global well-posedness for Schr\"odinger equations with derivative }, 
SIAM J. Math. Anal. {\bf 33 }, no.3, pp. 649-669  [2001].

\bibitem{CKSTT2} Colliander, J., Keel, M., Staffilani, G., Takaoka, H., and Tao, T.,  
{\em A refined global well-posedness for Schr\"odinger equations with derivative}, 
SIAM J. Math. Anal.  {\bf 34 }, no. 1, pp. 64-86  [2002].

\bibitem{CoOh} Colliander, J. and Oh, T., {\em Almost sure well-posedness of the cubic nonlinear Schr\"dinger equation below $L^2(\mathbb{T})$}, Duke Math. J. {\bf 161}, no. 3, pp. 367-414 [2012].

\bibitem{DPD} Da Prato, G. and Debussche, A. {\em Two-dimensional Navier-Stokes equations driven by a space-time white noise}, J. Funct. Anal. {\bf 196}, no. 1, pp. 180-210 [2002].


\bibitem{DPD2}  Da Prato, G. and Debussche, A. {\em Strong solutions to the stochastic quantization equations.}  Ann. Probab. {\bf 31}, no. 4, 1900-1916 [2003].

\bibitem{Deng} Deng, Y.,  {\em Invariance of the Gibbs measure for the Benjamin-Ono equation},  J. Eur. Math. Soc. (JEMS) {\bf 17}, no. 5, pp. 1107--1198  [2015].

\bibitem{Deng2} Deng, Y., {\em Two-dimensional nonlinear Schr\"odinger equation with random radial data}, Anal. PDE {\bf 5}, no. 5, pp.913-960 [2012].

\bibitem{DLM} Dodson, B., L\"uhrmann, J. and Mendelson, D., {\em Almost sure local well-posedness and scattering for the 4D cubic nonlinear Schr\"odinger equation}, Advances in Mathematics {\bf 347}, pp. 619-676 [2019].

\bibitem{GrNa}  Grigoryan, V. and Nahmod, A., {\em  Almost critical well-posedness for nonlinear wave equations with $Q_{\mu\nu}$ null
forms in 2D}, Math. Res. Lett.  {\bf 21} no. 2, 313Ð332 [2014].

\bibitem{Gr} Gr\"unrock, A., {\em Bi and Trilinear Schr\"odinger estimates in one space dimension with applications to cubic NLS and DNLS}, Int. Math. Res. Not. (IMRN) {\bf 41}, pp. 2525--2558 [2005].

\bibitem{Gr-wave} Gr\"unrock, A., {\em On the wave equation with quadratic nonlinearities in three space dimensions}, J. Hyperbolic
Differ. Equ. {\bf 8} no. 1, 1Ð8  [2011].

\bibitem{GH} Gr\"{u}nrock A. and Herr S.,  {\em Low regularity local well-posedness of the derivative nonlinear Schrödinger equation with periodic initial data},   SIAM J. Math. Anal. {\bf 39}, no. 6, pp. 1890-1920 [2008].


\bibitem{GIP}  Gubinelli, M., Imkeller P.  and Perkowski, N.  {\em Paracontrolled distributions and singular PDEs},  Forum Math Pi {\bf 3}, e6, 75 pp. [2015].

\bibitem{GIP2} Gubinelli, M., Imkeller, P. and Perkowski, N., {\em A Fourier analytic approach to pathwise stochastic integration}, Electron. J. Probab. {\bf 21}, Paper No. 2, 37 pp. [2016]. 

\bibitem{GP}  Gubinelli, M. and Perkowski, N. {\em Lectures on singular stochastic PDEs}, Ensaios Matem\'aticos [Mathematical Surveys], 29. Sociedade Brasileira de Matem\'atica, Rio de Janeiro. 89 pp. [2015].

\bibitem{GKO} Gubinelli, M., Koch, H. and Oh, T. {\em Renormalization of the two-dimensional stochastic nonlinear wave equations}, Trans. Amer. Math. Soc. {\bf 370}, no. 10, pp. 7335-7359 [2018].

\bibitem{GKO2} Gubinelli, M., Koch, H. and Oh, T. {\em Paracontrolled approach to the three-dimensional stochastic nonlinear wave equation with quadratic nonlinearity}, 
 \href{https://arxiv.org/abs/1811.07808}{arXiv:1811.07808} [math.AP].

\bibitem{Hairer} Hairer, M.,  {\em Solving the KPZ equation}, Ann. of Math. (2) {\bf 178}(2), pp. 559-664 [2013].

\bibitem{Hairer1}  Hairer, M., {\em  A theory of regularity structures}, Inventiones Math. {\bf 198}, (2), pp. 269-504, [2014].

\bibitem{Hairer2}  Hairer, M., {\em  Singular Stochastic PDE.} Proceedings of the ICM-Seoul, Vol. {\bf I}, 685-709 [2014].

\bibitem{Hairer3}  Hairer, M., {\em Regularity structures and the dynamical $\Phi^4_3$ model}, Current Developments in Mathematics 2014,  1-49, Int. Press, Somerville, MA  [2016]. 

\bibitem{Hay} Hayashi, N., {\em The initial value problem for the derivative nonlinear 
Schr\"odinger equation in the energy space}, 
Nonlinear Analysis {\bf 20} no. 7, pp. 823-833 [1993].

\bibitem{HayOz1} Hayashi, N. and Ozawa, T., {\em  On the derivative nonlinear 
Schr\"odinger equation}, Physica D {\bf 55} no. 1-2, pp. 14-36 [1992].

\bibitem{HayOz2} Hayashi, N. and Ozawa, T., {\em  Finite energy solutions of nonlinear   
Schr\"odinger equation of derivative type}, 
SIAM J. of Math. Anal. {\bf 25} no. 6, pp. 1488-1503 [1994].

\bibitem{Herr} Herr, S.,  {\em On the Cauchy problem for the derivative nonlinear 
Schr\"odinger equation with periodic boundary condition}, Int. Math. Res. Not. 
(IMRN), Article ID 96763,  pp. 1-33, [2006].

\bibitem{HiyOk} Hirayama, H. and Okamoto, M., {\em Random data Cauchy problem for the nonlinear Schr\"odinger equation with derivative nonlinearity}, Discrete Contin. Dyn. Syst. {\bf 36}, no. 12, pp. 6943-6974 [2016].

\bibitem{KaupNew} Kaup, D.J. and Newell, A.C., {\em An exact  solution for the derivative nonlinear Schr\"odinger equation}, J. of Math. Physics {\bf 19} no. 4, pp. 798-801 [1978].

\bibitem{Hormander}  H\"ormander, L.,  {\em The Analysis of Linear Partial Differential 
Operators. II}, Grundlehren Math. Wiss. {\bf 257}, Springer-Verlag, Berlin, [1983].

\bibitem{Jenkins1} Jenkins, R., Liu, J., Perry, P. and Sulem, C., {\em Global Well-Posesedness for the Derivative Nonlinear Schr\"odinger Equation}, 
\href{https://arxiv.org/abs/1710.03810}{arXiv:1710.03810}	 [math.AP].

\bibitem{Jenkins2} Jenkins, R., Liu, J., Perry, P. and Sulem, C., {\em Global Existence for the Derivative Nonlinear Schr\"odinger Equation with Arbitrary Spectral Singularities}, 
\href{https://arxiv.org/abs/1804.01506v2}{arXiv:1804.01506}	 [math.AP].


\bibitem{MWX} Miao, C., Wu, Y., and  Xu, G., 
{\em  Global well-posedness for Schr\"odinger equation with derivative in $H^{\frac{1}{2}}(\R)$}, J. Differential Equations {\bf 251}, no. 8, pp. 2164-2195 [2011]. 

\bibitem{Mos}  Mosincat, R., {\em Global well-posedness of the derivative nonlinear Schr\"odinger equation with periodic boundary condition in $H^{\frac{1}{2}}$}, J. Differential Equations {\bf 263}, no. 8, 4658-4722 [2017].

\bibitem{MoYo}  Mosincat, R. and Yoon, H., {\em Unconditional uniqueness for the derivatve Schr\"odinger equation on the real line},
\href{https://arxiv.org/abs/1810.09806}{arXiv:1810.09806}	 [math.AP].

\bibitem{MoOh} Mosincat, R. and Oh, T.,  {\em A remark on global well-posedness of the derivative nonlinear Schr\"odinger equation on the circle}, C. R. Math. Acad. Sci. Paris {\bf 353}, no. 9, pp. 837-841 [2015]. 

\bibitem{MouWe} Mourrat, J. and Weber, H. {\em The dynamic $\Phi^4_3$ model comes down from infinity}, Comm. Math. Phys. {\bf 356}, no. 3, pp. 673Ð753 [2017]. 


\bibitem{MouWeX} Mourrat, J.,  Weber, H. and Xu, W. {\em Construction of $\Phi^4_3$ diagrams for pedestrians}, From particle systems to partial differential equations, 1-46, Springer Proc. Math. Stat., 209, Springer,  [2017].

\bibitem{NORS} Nahmod, A. R., Oh, T., Rey-Bellet, L. and Staffilani, G., {\em Invariant weighted Wiener measures and  almost sure global well-posedness  for the periodic derivative NLS.} \,  J. Eur. Math. Soc. (JEMS) {\bf 14}, no. 4, 1275-1330, [2012].

\bibitem{NRSS} Nahmod, A. R.,  Rey-Bellet, L., Sheffield, S. and Staffilani, G., {\em   Absolute continuity of Brownian bridges under certain gauge transformations}. Math. Res. Lett. {\bf 18}, no. 5,  875-887 [2011]. 

\bibitem{NS} Nahmod, A. and Staffilani, G., {\em Almost sure well-posedness for the periodic 3D quintic nonlinear Schr\"odinger equation below the energy space}, J. Eur. Math. Soc. (JEMS) {\bf 17}, no. 7, pp. 1687-1759 [2015].

\bibitem{Oz}  Ozawa, T., {\em  On the nonlinear Schr\"odinger equations of derivative type}, Indiana Univ. Math. J. {\bf 45} no. 1, pp. 137-163 [1996].

\bibitem{Pelin}  Pelinovsky, D. E. and Shimabukuro, Y., {\em Existence of global solutions to the derivative NLS equation with the inverse scattering transform method} \href{https://arxiv.org/abs/1602.02118v1}{arXiv:1602.02118}	 [math.AP].


\bibitem{SS} Sulem, C and Sulem, J.P., {\em The nonlinear Sch\"odinger equation.  
Self-focusing and wave collapse.},  Applied Mathematical Sciences, {\bf 139}. Springer-Verlag, New York, [1999].

\bibitem{Tak1} Takaoka, H.,  
{\em Well-posedness for the one dimensional nonlinear Schr\"odinger equation with the derivative nonlinearity}, 
Adv. Differential Equations,  {\bf 4 }, no.4 pp. 561-580  [1999].

\bibitem{Tak2} Takaoka, H.,  
{\em Global well-posedness for Schr\"odinger equations with derivative in a nonlinear term 
and data in low order Sobolev spaces}, 
Electron. J. Differential Equations ,  no. 42, 23 pp. (electronic)  [2001].

\bibitem{Th} Thomann, L., {\em Random data Cauchy problem for supercritical Schr\"odinger equations}, Ann. Inst. H. Poincar\'e Anal. Non Lin\'eaire {\bf 26}, no. 6, 2385Ð2402 [2009].

\bibitem{TF} Tsutsumi, M. and Fukuda, I., {\em On solutions of the derivative nonlinear Schr\"odinger equation II}, Funkcial. Ekvac. {\bf 24}, no. 1,  pp. 85-94 [1981]. 

\bibitem{VarVe} Vargas, A. and Vega, L. {\em Global well-posedness for 1D Schr\"odinger equations for data with an infinite $L^2$-norm}, J. Math. Pures Appl. (9)  {\bf 80}, no. 10, 1029-1044, [2001]

\bibitem{Win} Win, Y., {\em Global Well-Posedness of the Derivative Nonlinear Schr\"odinger 
Equations on $\T$}, Funkcial. Ekvac. {\bf 53}, pp. 51-88, [2010]. 

\bibitem{Yue} Yue, H., {\em Almost sure well-posedness for the cubic nonlinear Schr\"odinger equation in the super-critical regime on $\mathbb{T}^d$, $d³3$}, \href{https://arxiv.org/abs/1808.00657}{arXiv:1808.00657} [math.AP].

\end{thebibliography}
\end{document}